\documentclass[12pt,leqno]{amsart}
\usepackage{amsmath,amsfonts,enumerate}%,comment
\usepackage{graphicx}
\usepackage{color}

\DeclareMathAlphabet{\mathpzc}{OT1}{pzc}{m}{it}

\newcommand{\ta}{\theta}
   \renewcommand{\b}{\beta}

\renewcommand{\d}{\delta}   \newcommand{\e}{\epsilon}

   \renewcommand{\l}{\lambda}

\newcommand{\newvec}[1]{\vbox{\ialign{##\crcr$\displaystyle\rightharpoonup$\crcr\noalign{\kern-1pt\nointerlineskip}
$\hfil\displaystyle{#1}\hfil$\crcr}}} %vector symbol with single alphabelt
\newcommand{\longrightharpoonup}{\relbar\joinrel\rightharpoonup}
\newcommand{\longvec}[1]{\vbox{\ialign{##\crcr$\displaystyle\longrightharpoonup$\crcr\noalign{\kern-1pt\nointerlineskip}
$\hfil\displaystyle{#1}\hfil$\crcr}}} %vector symbol with digraph

\newtheorem{theorem}{Theorem}[section]

\newtheorem{definition}[theorem]{Definition}

\newtheorem{lemma}[theorem]{Lemma}

\newtheorem{proposition}[theorem]{Proposition}

\newtheorem{remark}{Remark}[section]
\numberwithin{equation}{section}
\newcommand{\baa}{\begin{array}}
\newcommand{\eaa}{\end{array}}

\def\R{\mathbb{R}}

\def\epsilon{\varepsilon}
\newcommand{\bi}{\begin{itemize}}
\newcommand{\ei}{\end{itemize}}
\newcommand{\ben}{\begin{enumerate}}
\newcommand{\een}{\end{enumerate}}

\newcommand\be{\begin{equation}}
\newcommand\ee{\end{equation}}
\newcommand\bea{\begin{eqnarray}}
\newcommand\eea{\end{eqnarray}}
\newcommand\beaa{\begin{eqnarray*}}
\newcommand\eeaa{\end{eqnarray*}}

\begin{document}
%%%%%%%%%%%%%%%%%%%%%%%%%%%%%%%%%%%%%%%%
\title[persistence of preys]{Persistence of preys in a diffusive three species predator-prey system with a pair of strong-weak competing preys}

\author[Y.-S. Chen]{Yu-Shuo Chen}
\address{Department of Mathematics, Tamkang University, Tamsui, New Taipei City 251301, Taiwan}
\email{formosa1502@gmail.com}

\author[T. Giletti]{Thomas Giletti}
\address{Institut Elie Cartan de Lorraine, UMR 7502, University of Lorraine, 54506 Vandoeuvre-l\`es-Nancy, France}
\email{thomas.giletti@univ-lorraine.fr}

\author[J.-S. Guo]{Jong-Shenq Guo}
\address{Department of Mathematics, Tamkang University, Tamsui, New Taipei City 251301, Taiwan}
\email{jsguo@mail.tku.edu.tw}

\thanks{Corresponding author: J.-S. Guo}



\begin{abstract}
We investigate the traveling wave solutions of a three-species system involving a single predator and a pair of strong-weak competing preys.
%The competition is such that, in the absence of the predator, the strong prey drives the weak prey to extinction.
Our results show how the predation may affect this dynamics. More precisely, we describe several situations where the environment is initially inhabited by the predator and by either one of the two preys.
When the weak competing prey is an aboriginal species, we show that there exist traveling waves where the strong prey invades the environment and either replaces its weak counterpart,
or more surprisingly the three species eventually co-exist.
Furthermore, depending on the parameters, we can also construct traveling waves where the weaker prey actually invades the environment initially inhabited by its strong competitor and the predator. In all those situations, we find the infimum of the set of admissible wave speeds; these results are sharp at least when the three species diffusive at the same speed, where we show that a traveling wave exists if and only its speed is larger or equal to this infimum.
%Our results on the existence of traveling waves are sharp, in the sense that we find the infimum of the set of admissible wave speeds in all those situations; when the three species diffuse at the same speed, wave speed in all those situations
%Finally, our results on the existence of traveling waves are sharp {\red in the equi-diffusion case}, in the sense that we find the minimal wave speed in all those situations {\red when three species diffuse at the same speed}.
\end{abstract}

%%%%%%%%%%%%%%%%%%%%%%%%%%%%%%%%%%%%%
\maketitle
\setlength{\baselineskip}{18pt}
%%%%%%%%%%%%%%%%%%%%%%%%%%%%%%%%%%%%%
\section{Introduction}
\setcounter{equation}{0}

In this paper we consider the following system
\be\label{main}
\begin{cases}
u_t = d_1u_{xx} + r_1u(1-u-kv-b_1w),x\in\R,~t>0,\\
v_t = d_2v_{xx} + r_2v(1-hu-v-b_2w),x\in\R,~t>0,\\
w_t = d_3w_{xx} + r_3w(-1+au+av-w),x\in\R,~t>0,
\end{cases}
\ee
in which $u(x,t)$ and $v(x,t)$ stand for the densities of two preys and $w(x,t)$ is the density of the predator at $(x,t)$, $d_i$ and $r_i$, $i=1,2,3$, are their diffusion coefficients and
intrinsic growth rates, respectively, $h$ and $k$ denote the interspecific competition coefficients of two preys, the carrying capacities of two preys are assumed to
be $1$, $b_1$ and $b_2$ are predation rates of $u$ and $v$, respectively, and the conversion rates for both preys are assumed to be $a$.

Throughout this paper, we always assume that the parameters $d_i$, $r_i$ ($i=1,2,3$), $h,k,a,b_1$ and $b_2$ are all positive such that
\be\label{cond}
a>1,\quad h<1<k.
\ee
In particular, the single predator $w$ cannot survive without feeding on the preys, yet it can live together with either of those two preys. Moreover, the preys are competing and, in the absence of the predator $w$, the prey~$v$ is the strong competitor and the prey $u$ is the weak competitor. However, both preys undergo a priori different predation rates, and therefore the presence of the predator may invert their roles and lead to new dynamics, as we shall show in our main results.

Throughout this work, we shall consider situations when the single predator $w$ is an aboriginal species, and one of the two preys $u$ and $v$ is aboriginal while the other is alien. Our aim is to see the role of the predator in the ecological system \eqref{main} and in particular on the persistence of these two preys. It turns out that the following scenarios can happen, depending on the parameters:
\begin{enumerate}
\item
three species can co-exist, no matter which prey is alien;
\item
the alien strong competitor $v$ can replace the aboriginal weak competitor $u$ to co-exist with the predator $w$;
\item more surprisingly, it seems that the alien weak competitor $u$ may replace the aboriginal strong competitor $v$ to co-exist with the predator $w$ if it is more resistant to predation.
\end{enumerate}

%%%%%%%%%%%%%%

\subsection{The ODE system}\label{sec:ODE}

Let us start with some preliminary study of solutions of the diffusionless ODE system. Aside from the trivial steady state $(0,0,0)$, and since $a>1$,
we always have the existence of the \textit{semi-co-existence} states $E_{*} = (0,v_{*},w_{*})$  and $E^{*} = (u^{*},0,w^{*})$, where
\bea
&& u^{*} := \frac{1+b_1}{1+ab_1},\quad w^{*} := \frac{a-1}{1+ab_1};\label{vprey-free}\\
&& v_{*} := \frac{1+b_2}{1+ab_2},\quad w_{*} := \frac{a-1}{1+ab_2}.\label{uprey-free}
\eea
Let us briefly describe the stability of these two constant states. To do so, we define
\beaa
\b_{*}:= 1-kv_{*}-b_1w_{*},\; \b^{*}:= 1-hu^{*}-b_2w^{*},%\label{vbeta}
\eeaa
or equivalently
\be\label{bul*}
\b_{*} = \frac{-b_1(a-1)+b_2(a-k)-(k-1)}{1+ab_2},%\label{veig};\\
\; \b^{*} = \frac{b_1(a-h)-b_2(a-1)+(1-h)}{1+ab_1}.%\label{ueig},
\ee
{It is then easy to see that $\b^*>0$ (resp. $\b_* >0 $) implies that $E^{*}$ (resp. $E_{*}$) is unstable in the ODE sense. On the other hand, if $\b^* < 0$ (resp. $\b_* <0$) then $E^{*}$ (resp. $E_{*}$) is stable instead, again in the ODE sense. From \eqref{bul*}, one can check that $\b^*>0$ if and only if
\be\label{positive}
b_2<\frac{a-h}{a-1}b_1+\frac{1-h}{a-1}.
\ee
Also, $\b_*>0$ if and only if
\be\label{positive2}
a > k \; \mbox{ and } \; b_2>\frac{a-1}{a-k}b_1+\frac{k-1}{a-k}.
\ee

Finally, there exists at most one more (positive) constant state, in which the three species co-exist. This co-existence state only exists in some parameter range,
 in particular, one needs $\triangle \neq 0$, where
\beaa
\triangle :=1-hk + ab_1(1-h) - ab_2(k-1).
\eeaa
By Cramer's rule, when $\triangle\neq 0$ and if the co-existence state $E_c:=(u_c,v_c,w_c)$ exists, then
\be\label{co-ex}
u_c = \frac{\triangle_{u}}{\triangle}>0,\; v_c = \frac{\triangle_{v}}{\triangle}>0,\;w_c=\frac{\triangle_{w}}{\triangle}>0,
\ee
where
\beaa
&& \triangle_{u} := - b_1(a-1) + b_2(a-k) - (k-1),\quad \triangle_{v} := b_1(a-h) - b_2(a-1) + (1-h),\\
&& \triangle_{w} := a(2-h-k)-(1-hk).
\eeaa
Assume the existence of the (unique) positive co-existence state $E_c$. %Then we must have either $\triangle>0$ or $\triangle<0$.
Since the characteristic equation associated with the linearized system around $E_c$ is given by $\l^3+a_2\l^2+a_1\l+a_0 = 0$, where
\beaa
&&{a_2:=r_1u_c+r_2v_c+r_3w_c,}\\
&&{a_1:=r_1r_2u_cv_c(1-hk)+r_1r_3u_cw_c(1+ab_1)+r_2r_3v_cw_c(1+ab_2),}\\
&&{a_0:=(r_1r_2r_3u_cv_cw_c)\triangle,}
\eeaa
it is clear that $E_c$ is unstable, if $\triangle<0$. Moreover, one can also check that $E_c$ is stable, if $\triangle>0$ and $hk<1$. We point out that if either $E^*$ or $E_*$ is stable, then $E_c$ (if exists) must be unstable in the ODE sense. Indeed, if $E^*$ is stable so that $\beta^*<0$, then $\triangle_v<0$ and so $\triangle <0$ (if $E_c$ exists) which implies that $E_c$ is unstable.

%{\blue Note that $\triangle<0$ if and only if
%\beaa
%b_2>\frac{1-h}{k-1}b_1+\frac{1-hk}{a(k-1)}.
%\eeaa}
Let us mention a special case when $E_c$ is stable, whose interest is that a particular Lyapunov function then exists, which we shall use to classify entire in time solutions of \eqref{main}; see Lemma~\ref{LE-entire-full} below. Precisely, if $E_c$ exists, i.e., \eqref{co-ex} holds, and if moreover
\begin{equation}\label{ODE_lyapu}
k \sqrt{\frac{b_2}{b_1}} + h \sqrt{\frac{b_1}{b_2}} < 2,
\end{equation}
then $E_c$ is stable while $E^*$ and $E_*$ are unstable. Indeed, \eqref{ODE_lyapu} can be rewritten as $h k <1$ and
$$\frac{2-hk - 2 \sqrt{1-hk}}{k^2} b_1 < b_2 < \frac{2-hk + 2 \sqrt{1-hk}}{k^2} b_1.$$
We first show that $E_c$ is stable, i.e., $\triangle >0$.
We assume by contradiction that $\triangle <0$. Then, for $E_c$ to exist, it must hold that $\triangle_u<0$, $\triangle_v<0$, and $\triangle_w<0$. To have $\triangle < 0$, we need
$$b_2  > b_1 \frac{1-h}{k-1}.$$
Recalling the previous inequality, it follows that
$$b_1 \frac{1-h}{k-1} <\frac{2-hk + 2 \sqrt{1-hk}}{k^2} b_1,$$
i.e.,
$$ g(k):= \frac{2-hk + 2 \sqrt{1-hk}}{k^2} \cdot \frac{k-1}{1-h} > 1.$$
The derivative of $g(k)$ has the same sign as
$$(2-k)(2-hk + 2 \sqrt{1-hk}) - h (k^2 -k) (1 + \frac{1}{\sqrt{1-hk}}),$$
which is negative for $k \in ( 2, 1/h)$, and decreasing with respect to $k \in (1,2)$. Also, for $k = 2-h$, it is equal to 0. It follows that $g(k)$ is increasing with respect to $k \in [1, 2-h]$, then decreasing. Furthermore, $g(2-h)=1$. Since $g(2-h)$ is the maximum of $g$, we have reached a contradiction. Hence $\triangle >0$ and $E_c$ is stable. The existence of $E_c$ also implies that $\triangle_u$, $\triangle_v$ and $\triangle_w$ are positive. It follows that \eqref{positive} and \eqref{positive2} hold, and in particular $E_*$ and $E^*$ are unstable.

%Lastly, we point out for future use that
%\bea
%&& v_c = v_{*} - u_c\frac{h+ab_2}{1+ab_2}<v_{*},\quad w_c = w_{*} + u_c\frac{a(1-h)}{1+ab_2}>w_{*},\label{vc}\\
%&& u_c = u^{*} - v_c\frac{k+ab_1}{1+ab_1}<u^{*},\quad w_c = w^{*} - v_c\frac{a(k-1)}{1+ab_1}<w^{*}.\label{uc}
%\eea

%%%%%%%%%%%%%%%%%%%%%%%%%%%%

\subsection{The notion of a traveling wave}

To see the persistence of preys, our approach is to study the traveling wave solutions connecting two appropriate constant states of system~\eqref{main}.
A solution of \eqref{main} is called a \textit{traveling wave} solution with speed $s$, if there exist positive functions $\{\phi_1,\phi_2,\phi_3\}$ defined on $\R$ such that
$u(x,t) = \phi_1(x+st)$, $v(x,t) = \phi_2(x+st)$ and $w(x,t) = \phi_3(x+st)$; here $\phi_j$, $j=1,2,3$, are the wave profiles and are assumed to converge at $\pm \infty$ to constant states to be specified below.

Letting $z:=x+st$ and substituting $(u,v,w)(x,t) = (\phi_1,\phi_2,\phi_3)(z)$ into \eqref{main}, we get that $(s,\phi_1,\phi_2,\phi_3)$ must satisfy the following system of equations:
\be\label{TWS}
\begin{cases}
  d_1\phi_1''(z) -s\phi_1'(z)+ r_1\phi_1(z)[ 1-\phi_1(z)-k\phi_2(z)-b_1\phi_3(z)]=0,\;z\in\R,\\
  d_2\phi_2''(z) -s\phi_2'(z)+ r_2\phi_2(z)[ 1-h\phi_1(z)-\phi_2(z)-b_2\phi_3(z)]=0,\;z\in\R,\\
  d_3\phi_3''(z) -s\phi_3'(z)+ r_3\phi_3(z)[-1+a\phi_1(z)+a\phi_2(z)-\phi_3(z)]=0,\;z\in\R,
\end{cases}
\ee
where the prime denotes the derivative with respect to $z$.

%%%%%%%%%%%%%%%%%%%%
Throughout this work, we shall consider several types of traveling waves which differ from each other by their limits as $z \to \pm \infty$. Up to the symmetric change of variables $x \leftarrow -x$, we can always assume that
$$s \geq 0,$$
and therefore we shall refer to the limit of $(\phi_1,\phi_2,\phi_3)$ at $-\infty$ as the \textit{invaded state} or \textit{unstable tail}, and to the limit at $\infty$ as the \textit{invading state} or \textit{stable tail}.

As we mentioned before, we shall assume that the predator is aboriginal and consider the two cases where either of the two preys co-exists with the predator. Therefore we shall assume at the unstable tail that either
\be\label{BCv}
\lim_{z\rightarrow -\infty} (\phi_1(z),\phi_2(z),\phi_3(z)) = E^{*}:= (u^{*},0,w^{*}),
\ee
or
\be\label{BCu}
\lim_{z\rightarrow -\infty} (\phi_1(z),\phi_2(z),\phi_3(z)) = E_{*}:= (0,v_{*},w_{*}).
\ee
On the other hand, depending on the parameters we shall face two situations where either the three species eventually co-exist,
or the aboriginal prey goes to extinction and is replaced by the alien prey. In the former case, we shall have the asymptotic boundary condition at the stable tail
\be\label{Bc}
\lim_{z\rightarrow\infty} (\phi_1(z),\phi_2(z),\phi_3(z)) = E_c:=(u_c,v_c,w_c).
\ee
In the latter case, one must distinguish whether the aboriginal prey is the weak or the strong one; that is, we shall have either \eqref{BCv} together with
\bea
&& \lim_{z\rightarrow \infty} (\phi_1(z),\phi_2(z),\phi_3(z)) = E_{*}:= (0,v_{*},w_{*}),\label{BCu+}
\eea
or \eqref{BCu} together with
\bea
&& \lim_{z\rightarrow \infty} (\phi_1(z),\phi_2(z),\phi_3(z)) = E^{*}:= (u^{*},0,w^{*}),\label{BCv+}
\eea
%Notice that, under the condition $a\le k$, we must have $\triangle_u<0$ and so $E_c$ exists only when $\triangle<0$.
%Hence there cannot be a traveling wave connecting $E^*$ to $E_c$, if $\beta^*>0$ and $a\le k$, since both $E^*$ and $E_c$ are unstable.\medskip

%%%%%%%%%%%%%%
In order to study the existence of traveling waves for the non-monotone system~\eqref{main}, we apply a two-fold method based on a construction of appropriate generalized upper-lower solutions, and on a Schauder's fixed point theorem (cf. \cite{M01,WZ01,LLM10, LR14}).
This method has been proved to be very successful, and we refer the reader to \cite{HZ03,LLR06,LLM10,DX12,L14,LR14,CGY17,ZJ17} for 2-component systems as well as \cite{HL14,LWZY15,Z17,BP18,GNOW20} for 3-component systems.
However, the construction of a suitable set of upper and lower solutions depends heavily on the system at issue and therefore it is rather nontrivial, which is one difficulty in applying this method. These generalized upper-lower solutions serve as the upper and lower bounds of the domain for an appropriate integral operator deduced from \eqref{TWS}. Their purpose is to ensure that the integral operator maps this domain into itself, and therefore that a fixed point exists by Schauder's fixed point theorem. Provided that it satisfies the appropriate asymptotic conditions at the tails, this fixed point provides a traveling wave solution of \eqref{main}.

Another difficulty is precisely to check that the wave profile obtained above satisfies the wanted stable tail limit, which requires different approaches depending on the invading state. One of the classical approaches for deriving this limit is the method of contracting rectangles (cf. \cite{HL14,CGY17,GNOW20}). However, it turns out that this method is not directly applicable to our problem. To overcome this difficulty in the case of the waves connecting the semi-co-existence states, we introduce a new idea of dimension reduction (see Section~\ref{sec:stable_stail} below), which is one of the main contributions of this work. In this new method, we only consider, instead of 3-d rectangles, a sequence of shrinking 2-d rectangles. Moreover, we need to derive a priori certain positive lower bounds on $\phi_2$ and $\phi_3$ at the stable tail, since the lower bounds of these two components of our constructed upper-lower solutions are not good enough to apply the method of contracting rec
 tangles.
On the other hand, to obtain the traveling waves connecting the co-existence state at the stable tail limit, another approach is needed and therefore we instead apply a Lyapunov argument.

%%%%%%%%%%%%%%%%%%%%%%%
The rest of this paper is organized as follows. First, our main results are described in Section 2. Next, the existence of solutions to \eqref{TWS} with either \eqref{BCv} or \eqref{BCu} is carried out in Section 3.
 Section 4 is devoted to the derivation of the stable tail limit. Then we deal with the non-existence of waves in Section 5.
Finally, in Section 6, we provide the detailed verification of upper-lower solutions constructed in Section 3.

%%%%%%

%%%%%%%%%%%%%%
\section{Main results}
\setcounter{equation}{0}

In this section, we shall present the main results obtained in this paper.
%%%%%%%%%%%%%%%%%
We recall that
\beaa
\b_{*}= 1-kv_{*}-b_1w_{*},\; \b^{*}= 1-hu^{*}-b_2w^{*},%\label{vbeta}
\eeaa
or equivalently
\beaa
\b_{*} = \frac{-b_1(a-1)+b_2(a-k)-(k-1)}{1+ab_2},%\label{veig};\\
\; \b^{*} = \frac{b_1(a-h)-b_2(a-1)+(1-h)}{1+ab_1}.%\label{ueig},
\eeaa
{We also recall that the sign of $\b_{*}$ (resp. $\b^{*}$) determines the stability of the semi-co-existence states $E_* = (0,v_*,w_*)$ (resp. $E^* = (u^*,0,w^*)$) in the ODE sense. In particular, both states may be admissible as the unstable tail limit of the traveling wave solution. This leads us to introduce
\beaa
s_{*} := 2\sqrt{d_1r_1\b_*},\; s^{*} := 2\sqrt{d_2r_2\b^*},
\eeaa
whenever they are well-defined. These may be understood as the linear invasion speeds into the respective states $E_*$ and $E^*$ whenever they are unstable.
By an analogy with the well-understood Fisher-KPP scalar equation, one may also expect these values to be the infimum wave speeds, and this shall be confirmed by our results.

Our first result deals with the situation when the strong competitor prey is the alien species, and either replaces the weak aboriginal prey, or eventually co-exists with both the other species. With our notation, this means that the state $E^*= (u^*,0,w^*)$ may be invaded by either $E_* = (0,v_*, w_*)$ or $E_c = (u_c, v_c,w_c)$.
\begin{theorem}\label{th:sc1}
Suppose that $\b^*>0$, i.e., \eqref{positive} holds. Assume further that
\bea
&& r_2 \b^* \ge r_1 [ k + b_1 (2a-1)]. \label{vr}
\eea
Then system \eqref{TWS} has a bounded positive solution $(\phi_1,\phi_2,\phi_3)$ satisfying the boundary condition~\eqref{BCv}, for $s>s^*$ provided that
\be\label{vd}
{d_2\geq \max\{ d_1,d_3\} , \quad r_2 \b^* \geq r_3;}
\ee
and for $s=s^*$ provided that
\be\label{vvvd0}
{\frac{d_3}{2}  < d_1 =d_2\leq d_3 , \quad r_2 \left( 2 - \frac{d_3}{d_2}  \right) \beta^* \geq r_3 .}
\ee
Moreover, $(\phi_1, \phi_2, \phi_3)$ satisfies~\eqref{BCu+} if $\b_* < 0$ and
\bea
&& a>\frac{1}{1-h}, \quad b_2<\frac{a(1-h)-1}{a(2a-1)}; \label{hb2}
\eea
while $(\phi_1,\phi_2,\phi_3)$ satisfies \eqref{Bc} if \eqref{co-ex} and \eqref{ODE_lyapu} are enforced.
\end{theorem}
Some of the conditions in Theorem~\ref{th:sc1} appear to be mostly technical. Setting aside such assumptions, Theorem~\ref{th:sc1} roughly states that, when the semi-co-existence state $E^*$ is unstable, then there exists a traveling wave for any speed larger than $s^*$ connecting $E^*$ to a stable tail limit, which must also be a stable state of the ODE system. In particular, depending on the sign of $\b_*$, the stable tail limit shall be either $E_*$ or $E_c$.
\begin{remark}\label{rem:th1}
Let us point out that all the situations in Theorem~\ref{th:sc1} can be encountered in some parameter ranges. Indeed, notice first that \eqref{vr}, \eqref{vd} and \eqref{vvvd0} are the only conditions on the intrinsic growth rates and the diffusivities. These are clearly achievable and we focus on the choice of coupling parameters $a$, $h$, $k$, $b_1$ and $b_2$.

Consider first the case of a traveling wave satisfying \eqref{BCv} and \eqref{BCu+}, i.e., connecting the two semi-co-existence states.
All conditions $\b^* >0$, $\b_* < 0$ and the second inequality in~\eqref{hb2} rewrite as upper bounds on $b_2$. The only remaining condition is $a > \frac{1}{1-h}$, which raises no compatibility issue.
Therefore, a traveling wave connecting these two semi-co-existence states clearly exists in some parameter range, typically when $b_2$ is small.

The other case, i.e., of a traveling wave connecting $E^*$ to $E_c$, is a bit more complicated, because the corresponding assumptions involve both upper and lower bounds on $b_2$. First, one may choose $a$, $h$ and $k$ so that $hk < 1$, $h+k < 2$ and $a > \frac{1-hk}{2 -h -k}$ hold. It follows that $\triangle_w >0$, and \eqref{co-ex} rewrites as
$$\triangle >0, \; \triangle_u  >0  , \; \triangle_v >0 .$$
The positivity of $\triangle_v$ also implies that $\beta^* >0$. On the other hand, the assumption $ k \sqrt{\frac{b_2}{b_1}} + h \sqrt{\frac{b_1}{b_2}}< 2$ rewrites as
$$\frac{2 - hk - 2 \sqrt{1 -hk}}{k^2} b_1 < b_2 < \frac{2 - hk + 2 \sqrt{1-hk}}{k^2} b_1 .$$
From the definitions of $\triangle$, $\triangle_u$ and $\triangle_v$ in Subsection~\ref{sec:ODE}, one can find $b_2$ such that a traveling wave connecting $E^*$ to $E_c$ exists if
\beaa
 && \max \left\{\frac{a-1}{a-k}b_1+\frac{k-1}{a-k},  \frac{2 - hk - 2 \sqrt{1 -hk}}{k^2} b_1 \right\} \\
&< &  \min \left\{ \frac{1-h}{k-1} b_1 + \frac{1 -hk}{a (k-1)} ,  \frac{a-h}{a-1}b_1+\frac{1-h}{a-1},  \frac{2 - hk + 2 \sqrt{1-hk}}{k^2} b_1 \right\}.
\eeaa
This is achievable, for instance, if $k$ is close to 1.
%
%{\blue
%The condition $\b_*<0$ can be reached, for examples,
%
%{\bf Case 1.} $1<a\le k$. Then $\b_*<0$ for any $b_1>0$ and $b_2>0$.
%
%{\bf Case 2.} $a>k$. Then we have $\b_*<0$ (for some $b_1>0$ and $b_2>0$), if
%\beaa%\label{negative}
%b_2<\frac{a-1}{a-k}b_1+\frac{k-1}{a-k}.
%\eeaa
%Conditions \eqref{c2} and \eqref{hb2} hold for some positive constants $b_1$ and $b_2$, if
%\beaa
%\frac{a-1}{a-k}b_1+\frac{k-1}{a-k}<\frac{a(1-h)-1}{a(2a-1)}.
%\eeaa
%This can be achieved, for example, if we assume $h+2k<3$, $a>a_+$ and $b_1\ll 1$, where $a_+$ is the larger root of
%\beaa
%(3-h-2k)a^2-(2-hk)a+k=0.
%\eeaa
%}
\end{remark}

Now we turn to the more surprising case when the strong competitor is the aboriginal prey, i.e., the unstable tail limit of the traveling wave is the semi-co-existence state $E_*$.
It turns out that, due to the predation, it is possible that the weak alien prey invades the environment with positive speed.
\begin{theorem}\label{th:cs2}
Suppose that $\b_*>0$, i.e., \eqref{positive2} holds. Assume further that
\bea
&&  r_1\b_{*}\geq r_2[h+b_2(2a-1)]. \label{uur} %\\
\eea
Then system \eqref{TWS} has a bounded positive solution $(\phi_1,\phi_2,\phi_3)$ satisfying the boundary condition \eqref{BCu}, for $s>s_*$ provided that
\be\label{uud}
d_1\geq \max\{ d_2,d_3\}, \quad r_1 \beta_* \geq r_3;
\ee
and for $s=s_*$ provided that
\be\label{uuud0}
\frac{d_3}{2}  < d_1 =d_2\leq d_3 , \quad r_1 \left( 2 - \frac{d_3}{d_1}  \right) \beta^* \geq r_3.
\ee
Moreover, we have that
$$\liminf_{z\to +\infty} \phi_i(z) >0\; \mbox{ for } i = 1,3.$$
Furthermore, if \eqref{co-ex} and \eqref{ODE_lyapu} are enforced, then $(\phi_1,\phi_2,\phi_3)$ satisfies \eqref{Bc}.
\end{theorem}

Theorem~\ref{th:cs2} tells us that a weak intruding prey can invade an environment inhabited by the strong competing prey, thanks to the effect of predation.
In the co-existence case, the three even species converge together to a positive equilibrium.
In the case when the co-existence state $E_c$ is unstable, we expect that the weak prey completely replaces the strong one, i.e., that $(\phi_1 , \phi_2 , \phi_3)$ satisfies \eqref{BCv+}.
Unfortunately we have not been able to prove this rigorously and we leave it as an open issue for future work.
\begin{remark}\label{rem:th2}
Let us again point out that the assumptions in Theorem~\ref{th:cs2} can indeed be satisfied. The argument is the same as in Remark~\ref{rem:th1}.
\end{remark}
%\begin{remark}
%{\blue The condition $\b^*<0$ holds if and only if
%\beaa%\label{negative2}
%b_2>\frac{a-h}{a-1}b_1+\frac{1-h}{a-1}.
%\eeaa
%Note that $\b^* < 0$ implies that $E_c$ is unstable.}
%%{\red
%%Conditions \eqref{c2} and \eqref{uub2} hold for some $b_1>0$ and $b_2>0$, if
%%\beaa
%%\frac{h}{a(1-h)-1}<\frac{a-h}{a-1}b_1+\frac{1-h}{a-1}.
%%\eeaa
%%}
%\end{remark}

\begin{remark}
We note from Theorem~\ref{th:sc1} that traveling waves exist for all speeds $s\ge s^*$ only when both conditions \eqref{vd} and \eqref{vvvd0} hold.
In particular, three species must diffuse at the same speed, i.e., $d_1=d_2=d_3$. The same limitation holds for the traveling waves obtained in Theorem~\ref{th:cs2}.
\end{remark}

Our last main result shows that the wave speeds $s^*$ and $s_*$ exhibited above are truly the infimum wave speeds of traveling wave solutions. % {\red when three species diffuse at the same speed}
More precisely:
\begin{theorem}\label{th:non}
The following statements hold:
\begin{enumerate}
\item
Assume that $\b^* >0$, hence $s^* >0$. Then no positive solutions of \eqref{TWS}, \eqref{BCv} and either \eqref{Bc} or \eqref{BCu+} exist for $s<s^*$.
\item
Assume that $\b_* >0$, hence $s_* >0$. Then no positive solutions of \eqref{TWS}, \eqref{BCu} and either \eqref{Bc} or \eqref{BCv+} exist for $s<s_*$.
\end{enumerate}
\end{theorem}
As we shall observe in Section~\ref{sec:non}, in the above non-existence theorem the stable tail limits can actually be replaced by the positivity of the infimum limit at $\infty$ of the alien species component. In particular, either prey invading the environment must do so at least at the corresponding speed $s^*$ or $s_*$.

%%%%%%%%%%%%%%%%%%%%%

%%%%%%%%%%%%%%%%%%%%%%%%%%%%%
\section{Existence of solutions to \eqref{TWS}}\label{existence}
\setcounter{equation}{0}

%\subsection{Upper and lower solutions}
First, we give the definition of generalized upper-lower solutions of \eqref{TWS} as follows.

\begin{definition}\label{uplodfn}
Nonnegative and continuous functions $(\overline{\phi}_1,\overline{\phi}_2,\overline{\phi}_3)$ and $(\underline{\phi}_1,\underline{\phi}_2,\underline{\phi}_3)$ are called a pair of generalized upper and lower solutions
 of \eqref{TWS} if ~$\overline{\phi}_i'',~\underline{\phi}_i'',~\overline{\phi}_i',~\underline{\phi}_i',~i=1,2,3$, are bounded functions and satisfy the following inequalities
\bea
&&\mathcal{U}_1(z):=d_1\overline{\phi}_1''(z) -s\overline{\phi}_1'(z)+r_1\overline{\phi}_1(z)[1-\overline{\phi}_1(z)-k\underline{\phi}_2(z)-b_1\underline{\phi}_3(z)] \leq 0,\label{u1}\\
&&\mathcal{U}_2(z):=d_2\overline{\phi}_2''(z) -s\overline{\phi}_2'(z)+r_2\overline{\phi}_2(z)[1-h\underline{\phi}_1(z)-\overline{\phi}_2(z)-b_2\underline{\phi}_3(z)] \leq 0,\label{u2}\\
&&\mathcal{U}_3(z):=d_3\overline{\phi}_3''(z) -s\overline{\phi}_3'(z)+r_3\overline{\phi}_3(z)[-1+a\overline{\phi}_1(z)+a\overline{\phi}_2(z)-\overline{\phi}_3(z)]    \leq 0,\label{u3}\\
&&\mathcal{L}_1(z):=d_1\underline{\phi}_1''(z) -s\underline{\phi}_1'(z)+r_1\underline{\phi}_1(z)[1-\underline{\phi}_1(z)-k\overline{\phi}_2(z)-b_1\overline{\phi}_3(z)] \geq 0,\label{l1}\\
&&\mathcal{L}_2(z):=d_2\underline{\phi}_2''(z) -s\underline{\phi}_2'(z)+r_1\underline{\phi}_2(z)[1-h\overline{\phi}_1(z)-\underline{\phi}_2(z)-b_2\overline{\phi}_3(z)] \geq 0,\label{l2}\\
&&\mathcal{L}_3(z):=d_3\underline{\phi}_3''(z) -s\underline{\phi}_3'(z)+r_3\underline{\phi}_3(z)[-1+a\underline{\phi}_1(z)+a\underline{\phi}_2(z)-\underline{\phi}_3(z)]\geq 0,\label{l3}
\eea
for $z\in\mathbb{R}\backslash E$ with some finite set $E = \{ z_1,z_2,\ldots,z_m\}$.
\end{definition}

%%%%%%%%%%%%
%\subsection{General framework for the existence of solutions to \eqref{TWS}}

%{\red First, we transform the differential system \eqref{TWS} into an integral system .......}
With this notion of generalized upper-lower solutions, we have the following existence theorem for system \eqref{TWS}.

\begin{proposition}\label{exist}
Given $s>0$. Suppose that system \eqref{TWS} has a pair of generalized upper-lower solutions $(\overline{\phi}_1,\overline{\phi}_2,\overline{\phi}_3)$ and
$(\underline{\phi}_1,\underline{\phi}_2,\underline{\phi}_3)$ such that
\bea
&&\underline{\phi}_i(z)\le\overline{\phi}_i(z),\;\forall\,z\in\R,\; i=1,2,3,\label{order}\\
&&\lim_{z\to z_j^+}\overline{\phi}_i'(z)\le\lim_{z\to z_j^-}\overline{\phi}_i'(z),\; \lim_{z\to z_j^-}\underline{\phi}_i'(z)\le\lim_{z\to z_j^+}\underline{\phi}_i'(z),\;\forall\,z_j\in E,\; i=1,2,3.\label{lrd}
\eea
Then system \eqref{TWS} has a solution $(\phi_1,\phi_2,\phi_3)$ such that $\underline{\phi}_i\le\phi_i\le\overline{\phi}_i$, $i=1,2,3$.
\end{proposition}

The proof of Proposition~\ref{exist} can be done by a standard argument as that in, e.g., \cite{M01,WZ01,HL14}), and thus we omit it here.

%%%%%%%%%%%%%%%%%%%%%%
\subsection{Upper-lower solutions for waves invading $E^* = (u^*,0,w^*)$}\label{vv}
In this subsection we shall construct generalized upper-lower solutions of \eqref{TWS} with boundary condition \eqref{BCv} at the unstable tail.
{Hence we assume that $\b^* >0 $ so that $E^*$ is unstable. Also, we impose condition \eqref{vr} % and \eqref{extra}
from Theorem~\ref{th:sc1}.}

\subsubsection{Case $s>s^*$} We fix here $s > s^*$ and further assume that \eqref{vd} is enforced. Let~$\l_1$ and~$\l_2$ be the two positive roots of
\beaa
G(x): = d_2x^2-sx+r_2\b^{*},
\eeaa
which are given by
\bea\label{l}
\l_1:=\frac{s-\sqrt{s^2-4d_2r_2\b^{*}}}{2d_2},\quad\l_2:=\frac{s+\sqrt{s^2-4d_2r_2\b^{*}}}{2d_2}.
\eea
It follows from the first inequality in~\eqref{vd} that
\be\label{vineq}
 0<\l_1\le \min{\left\{\frac{s}{2d_1},\frac{s}{2d_3}\right\}}.
\ee
Moreover, by \eqref{vr} and \eqref{vd}, we have that
$$d_1 \lambda_1^2 - s \lambda_1 + r_1 [k + b_1 (2a -1)] \leq d_2 \lambda_1^2 - s\lambda_1 + r_2 \beta^* = 0,$$
hence
\be\label{vineq2}
0<R:=\frac{r_1[k+b_1(2a-1)]}{-(d_1\l_1^2-s\l_1)}\le 1.
\ee

Now we introduce the following continuous functions %$\overline{\phi}_j(z)$ and $\underline{\phi}_j(z)$ for $j=1,2,3$.
\bea
&& \overline{\phi}_1(z) =\left\{
                           \begin{array}{ll}
                               u^{*}+b_1w^{*}e^{\l_1z}, &\hbox{$z < 0$,} \\
                               1                       , &\hbox{$z > 0$,}
                           \end{array}
                           \right.\label{vsols11}\\
&& \underline{\phi}_1(z) =\left\{
                           \begin{array}{ll}
                               u^{*}(1-p_1e^{\l_1 z}), &\hbox{$z < z_1$,} \\
                               0                  , &\hbox{$z > z_1$,}
                           \end{array}
                           \right.\label{vsols12}\\
&& \overline{\phi}_2(z) =\left\{
                           \begin{array}{ll}
                               e^{\l_1 z}, &\hbox{$z < 0$,} \\
                               1        ,  &\hbox{$z > 0$,}
                           \end{array}
                           \right.\label{vsols21}\\
&& \underline{\phi}_2(z) =\left\{
                           \begin{array}{ll}
                               e^{\l_1 z}-qe^{\mu\l_1 z}, &\hbox{$z < z_2$,} \\
                               0         , &\hbox{$z > z_2$,}
                           \end{array}
                           \right.\label{vsols22}\\
&& \overline{\phi}_3(z) =\left\{
                           \begin{array}{ll}
                               w^{*}+Ae^{\l_1 z}, &\hbox{$z <0$,} \\
                               2a-1         , &\hbox{$z >0$,}
                           \end{array}
                           \right.\label{vsols31}\\
&& \underline{\phi}_3(z) =\left\{
                           \begin{array}{ll}
                               w^{*}(1 -e^{\l_1z}), &\hbox{$z<0$,} \\
                               0                  , &\hbox{$z>0$,}
                           \end{array}
                           \right.\label{vsols32}
\eea
where constants $A,p_1,\mu$ and $q$ are defined in sequence as follows:
\bea
&& A = (2a-1)-w^{*} >0;\label{vB}\\
&& R\le p_1\le 1;\label{vp}\\
&& 1<\mu<\min{ \left\{ 2,\l_2/\l_1 \right\} };\label{vmu}\\
&& q>\max{\left\{  1 ,\frac{r_2(hb_1w^{*}+1+b_2A)}{-G(\mu\l_1)} \right\} } . \label{vq}
%&& \d_2\in \left(0,\min \left\{1-h-b_2(2a-1),\max_{z\in\mathbb{R}}\{e^{\l_1 z}-qe^{\mu\l_1 z}\}/2 \right\} \right).\label{vdd}
\eea
The points $z_1$ and $z_2$ are defined by
\beaa
 z_1:=\frac{-\ln{p_1}}{\l_1},\;  z_2 :=  \frac{-\ln(q)}{(\mu-1)\l_1 } .
\eeaa
By the choice of $p_1$  in~\eqref{vp}, which is admissible due to \eqref{vineq2}, and because $q>1$, we have that $z_2<0\le z_1$. Note also $G(\mu\l_1)<0$ so that $q$ is well-defined.
%The points $z_1$ and $z_2\in(z_M,z_0)$ are defined by
%\beaa
% z_1:=\frac{-\ln{p_1}}{\l_1},\;e^{\l_1 z_2}-qe^{\mu\l_1 z_2} = \delta_2,
%\eeaa
%where $z_M:=-\ln(q\mu)/[(\mu-1)\l_1]$ and $z_0:=-\ln(q)/[(\mu-1)\l_1]<0$. By the choice of $p_1$  in~\eqref{vp}, which is admissible due to \eqref{vineq2}, and because $q>1$, we have that $z_2<0\le z_1$.
%Also, by the choice of $b_2$ in \eqref{hb2}, the constant $\d_2$ in \eqref{vdd} is positive.
%Note that $G(\mu\l_1)<0$ so that $q$ is well-defined.
%

\begin{lemma}\label{u0w}
Suppose that $\beta^*>0$ and $s>s^{*}$. Let %\eqref{hb2},
\eqref{vr} %, {\eqref{extra}}
and \eqref{vd} be enforced.
Then the functions $(\overline{\phi}_1,\overline{\phi}_2,\overline{\phi}_3)$
and $(\underline{\phi}_1,\underline{\phi}_2,\underline{\phi}_3)$ defined in \eqref{vsols11}-\eqref{vsols32}
are a pair of generalized upper and lower solutions of \eqref{TWS} in the sense of Definition~\ref{uplodfn}, satisfy \eqref{order}-\eqref{lrd} and are such that boundary condition \eqref{BCv} holds at the unstable tail.
\end{lemma}
The proof of Lemma~\ref{u0w} is given in Section~\ref{verul}.

%%%%%%%%%%%%%%%%%%%%%
\subsubsection{Case $s=s^{*}$}
{
When $s=s^{*}$, then $\l_1 =\l_2= s/(2d_2)$. Here we impose condition \eqref{vvvd0}.
%\be\label{vvvd}
% d_1=d_2=d,\quad\frac{d_3}{2}\leq d\leq d_3.
%\ee
Then we introduce the following continuous functions
\bea
&& \overline{\phi}_1(z) =\left\{
                           \begin{array}{ll}
                               u^{*}+L^{*}b_1w^{*}(-z)e^{\l_1z}, &\hbox{$z < -2/\l_1$,}\\
                               1                               , &\hbox{$z > -2/\l_1$,}
                           \end{array}
                           \right.\label{vvvsols11}\\
&& \underline{\phi}_1(z) =\left\{
                           \begin{array}{ll}
                               u^{*}[1-p_1L^{*}(-z)e^{\l_1 z}], &\hbox{$z < z_1$,}\\
                               0                              , &\hbox{$z > z_1$,}
                           \end{array}
                           \right.\label{vvvsols12}\\
&& \overline{\phi}_2(z) =\left\{
                           \begin{array}{ll}
                               L^{*}(-z)e^{\l_1 z}, &\hbox{$z < -2/\l_1$,}\\
                               1                  , &\hbox{$z > -2/\l_1$,}
                           \end{array}
                           \right.\label{vvvsols21}\\
&& \underline{\phi}_2(z) =\left\{
                           \begin{array}{ll}
                               [L^{*}(-z)-q(-z)^{1/2}]e^{\l_1 z}, &\hbox{$z < z_2$,}\\
                               0                    , &\hbox{$z > z_2$,}
                           \end{array}
                           \right.\label{vvvsols22}\\
&& \overline{\phi}_3(z) =\left\{
                           \begin{array}{ll}
                               w^{*}+L^{*}A(-z)e^{\l_1 z}, &\hbox{$z<-2/\l_1$,} \\
                               2a-1                      , &\hbox{$z>-2/\l_1$,}
                           \end{array}
                           \right.\label{vvvsols31}\\
&& \underline{\phi}_3(z) =\left\{
                           \begin{array}{ll}
                               w^{*}[1 -L^{*}(-z)e^{\l_1z}], &\hbox{$z<-2/\l_1$,} \\
                               0                           , &\hbox{$z>-2/\l_1$,}
                           \end{array}
                           \right.\label{vvvsols32}
\eea
where $L^{*} := \l_1e^2/2$, $A=2a-1-w^*$, and the constants $p_1, q$ are chosen in sequence such that
\be\label{vp_bis}
\max \{ R , 2 e^{-1} \}  \leq p_1 \leq 1
\ee
(notice that $R \le 1$ still holds thanks to \eqref{vr} and $d_1 = d_2$), and
\be\label{vvq}
q  > \max{ \left\{ \frac{4r_2(L^{*})^2M(hb_1w^{*}+1+b_2 A)}{d_2},L^{*}\sqrt{\frac{2}{\l_1}}  \right\} },\; \mbox{ with } \; M:=\left(\frac{7}{2\l_1e}\right)^{7/2}.
\ee
%and $\d_2$ is defined by
%\be\label{vdd*}
% \d_2\in \left(0,\min \left\{1-h-b_2(2a-1),\max_{z\le 0} \{[L^{*}(-z)-q(-z)^{1/2}]e^{\l_1 z}\}/2 \right\} \right).
%\ee
{It is easy to check that the function $z \in ( -\infty,0] \mapsto p_1 L^* (-z) e^{\lambda_1 z}$ reaches its maximum at $-1/\lambda_1$, where it takes the value $p_1 e/2 \geq 1$. Thus we can define $z_1$ by
\be
p_1 L^* (-z_1) e^{\l_1 z_1} = 1, \quad z_1 \in \left[ -\frac{2}{\l_1}, -\frac{1}{\l_1} \right].
\ee
We also define
$$z_2 := - \left( \frac{q}{L^*} \right)^2.$$
%Similarly,} the function $z \in (-\infty,0] \mapsto [L^{*}(-z)-q(-z)^{1/2}]e^{\l_1 z}$, has a unique maximal point $\hat{z}\in(-\infty,z_0)$, with $z_0:=-(q/L^{*})^2$. Then $z_2$ is defined uniquely by
%\beaa
% [L_{*}(-z_2)-q(-z_2)^{1/2}]e^{\l_1 z_2} = \d_2 , \quad  z_2\in(\hat{z},z_0).
%\eeaa
Note that the choice of $q$ in \eqref{vvq} ensures that $z_2<-2/\l_1$.

Then we have the following lemma, whose proof is given in Section~\ref{verul}.
\begin{lemma}\label{u00w}
Suppose that $\beta^*>0$ and $s=s^{*}$. Let %\eqref{hb2},
\eqref{vr} and \eqref{vvvd0} be enforced.
Then the functions $(\overline{\phi}_1,\overline{\phi}_2,\overline{\phi}_3)$
and $(\underline{\phi}_1,\underline{\phi}_2,\underline{\phi}_3)$ defined in \eqref{vvvsols11}-\eqref{vvvsols32}
are a pair of generalized upper and lower solutions of \eqref{TWS} in the sense of Definition~\ref{uplodfn}, satisfy \eqref{order}-\eqref{lrd} and are such that boundary condition \eqref{BCv} holds at the unstable tail.
\end{lemma}
With Lemmas~\ref{u0w} and \ref{u00w} in hand, the first part of Theorem~\ref{th:sc1} is proved by applying Proposition~\ref{exist}.

%%%%%%%%%%%%%%%%%%%%%%
\subsection{Upper-lower solutions for waves invading $E_* = (0,v_*,w_*)$}\label{uu}

In this subsection we shall construct generalized upper-lower solutions of \eqref{TWS} with boundary condition \eqref{BCu}. Hence we assume that $\b_*>0$ and $E_*$ is unstable, and we impose condition \eqref{uur}.

\subsubsection{Case  $s>s_{*}$} Here we also impose condition \eqref{uud}. Let $\sigma_1$ and $\sigma_2$ be the two positive roots of $H(x): = d_1x^2-sx+r_1\b_{*}$, that is,
\be\label{ul}
\sigma_1:=\frac{s-\sqrt{s^2-4d_1r_1\b_{*}}}{2d_1},\quad\sigma_2:=\frac{s+\sqrt{s^2-4d_1r_1\b_{*}}}{2d_1}.
\ee
By \eqref{uud}, we have
\be\label{uineq}
0<\sigma_1\le \min{\left\{\frac{s}{2d_2},\frac{s}{2d_3}\right\}}.
\ee
Moreover, by the same reasoning as that for \eqref{vineq2}, using \eqref{uur} and \eqref{uud}, we have
\be\label{uineq2}
0<S:=\frac{r_2[h+b_2(2a-1)]}{-(d_2\sigma_1^2-s\sigma_1)}\le 1.
\ee
Let the constants $B,p_2,\mu,q$ be defined in sequence as follows
\bea
&&  B = (2a-1)-w_{*};\label{uuA}\\
&& S\le p_2\le 1;\label{uup}\\
%&& \gamma_3=h/b_2,\; p_3=1-\gamma_3/w_*;\label{gp3}\\
&& 1<\mu<\min\{ 2,\sigma_2/\sigma_1 \};\label{uumu}\\
&& q>\max\left\{1 ,\frac{r_1(1+kb_2w_{*}+b_1B)}{-H(\mu\sigma_1)} \right\};\label{uuq} %\\
%&& \gamma_1\in(0,\min\{1-k(1-h)-b_1(2a-1),\max_{z\in\mathbb{R}}\{ e^{\sigma_1 z}-qe^{\mu\sigma_1z}\}/2\}). \label{uudd}
\eea
as well as $z_0:=-\ln(q)/[(\mu-1)\sigma_1]$, $z_2:={-\ln(p_2)}/{\sigma_1}$. Note that $H(\mu\sigma_1)<0$ and so $q$ is well-defined. Also, we have $z_0<0\le z_2$.

With these parameters, we introduce the following continuous functions
\bea
&& \overline{\phi}_1(z) =\left\{
                           \begin{array}{ll}
                               e^{\sigma_1z}, &\hbox{$z < 0$,} \\
                               1        , &\hbox{$z > 0$,}
                           \end{array}
                           \right.\label{uusols11}\\
&& \underline{\phi}_1(z) =\left\{
                           \begin{array}{ll}
                               e^{\sigma_1 z}-qe^{\mu\sigma_1z}, &\hbox{$z < z_0$,} \\
                               0              , &\hbox{$z > z_0$,}
                           \end{array}
                           \right.\label{uusols12}\\
&& \overline{\phi}_2(z) =\left\{
                           \begin{array}{ll}
                              v_{*}+b_2w_{*}e^{\sigma_1 z}, &\hbox{$z < 0$,} \\
                               1                          , &\hbox{$z > 0$,}
                           \end{array}
                           \right.\label{uusols21}\\
&& \underline{\phi}_2(z) =\left\{
                           \begin{array}{ll}
                               v_{*}(1-p_2e^{\sigma_1 z}), &\hbox{$z < z_2$,} \\
                               0                     , &\hbox{$z > z_2$,}
                           \end{array}
                           \right.\label{uusols22}\\
&& \overline{\phi}_3(z) =\left\{
                           \begin{array}{ll}
                               w_{*}+Be^{\sigma_1 z}, &\hbox{$z <0$,} \\
                               2a-1             , &\hbox{$z >0$,}
                           \end{array}
                           \right.\label{uusols31}\\
&& \underline{\phi}_3(z) =\left\{
                           \begin{array}{ll}
                                w_{*}(1 -e^{\sigma_1z}), &\hbox{$z<0$,} \\
                                0                     , &\hbox{$z>0$.}
                           \end{array}
                           \right.\label{uusols32}
\eea
Then we have the following lemma, whose proof is given in Section~\ref{verul}.
\begin{lemma}\label{0vw}
Suppose that $\beta_*>0$ and $s>s_{*}$. {Let \eqref{uur} and \eqref{uud} be enforced.}
Then the functions $(\overline{\phi}_1,\overline{\phi}_2,\overline{\phi}_3)$
and $(\underline{\phi}_1,\underline{\phi}_2,\underline{\phi}_3)$ defined in \eqref{uusols11}-\eqref{uusols32}
are a pair of generalized upper and lower solutions of \eqref{TWS} in the sense of Definition~\ref{uplodfn}, satisfy \eqref{order}-\eqref{lrd} and are such that boundary condition \eqref{BCu} holds at the unstable tail.
\end{lemma}

%%%%%%%%%%%%%%%%%%%%%%
\subsubsection{Case $s=s_{*}$}

{Here we impose \eqref{uuud0}.} When $s=s_{*}$, then $\sigma_1 =\sigma_2= s/(2d_1)=s/(2d_2)$.
Then we introduce the following continuous functions $\overline{\phi}_j(z)$ and $\underline{\phi}_j(z)$ for $j=1,2,3$.
\bea
&& \overline{\phi}_1(z) =\left\{
                           \begin{array}{ll}
                               L_{*}(-z)e^{\sigma_1z}, &\hbox{$z < -2/\sigma_1$,}\\
                               1               , &\hbox{$z > -2/\sigma_1$,}
                           \end{array}
                           \right.\label{uuusols11}\\
&& \underline{\phi}_1(z) =\left\{
                           \begin{array}{ll}
                               [L_{*}(-z)-q(-z)^{1/2}]e^{\sigma_1 z}, &\hbox{$z < z_0$,}\\
                              0                     , &\hbox{$z > z_0$,}
                           \end{array}
                           \right.\label{uuusols12}\\
&& \overline{\phi}_2(z) =\left\{
                           \begin{array}{ll}
                               v_{*}+L_{*} b_2w_{*} (-z)e^{\sigma_1 z}, &\hbox{$z < -2/\sigma_1$,}\\
                              1                               , &\hbox{$z > -2/\sigma_1$,}
                           \end{array}
                           \right.\label{uuusols21}\\
&& \underline{\phi}_2(z) =\left\{
                           \begin{array}{ll}
                               v_{*}[1-p_2L_{*}(-z)e^{\sigma_1 z}], &\hbox{$z < z_2$,}\\
                               0                            , &\hbox{$z > z_2$,}
                           \end{array}
                           \right.\label{uuusols22}\\
&& \overline{\phi}_3(z) =\left\{
                           \begin{array}{ll}
                               w_{*}+L_{*}B(-z)e^{\sigma_1 z}, &\hbox{$z<-2/\sigma_1$,} \\
                               2a-1                      , &\hbox{$z>-2/\sigma_1$,}
                           \end{array}
                           \right.\label{uuusols31}\\
&& \underline{\phi}_3(z) =\left\{
                           \begin{array}{ll}
                               w_{*}[1 -L_{*}(-z)e^{\sigma_1z}], &\hbox{$z<-2/\sigma_1$,} \\
                               0                      , &\hbox{$z>-2/\sigma_1$,}
                           \end{array}
                           \right.\label{uuusols32}
\eea
where $L_{*} := \sigma_1e^2/2$, $B=(2a-1)-w_{*}$, $p_2$ satisfies $\max \{ S, 2 e^{-1} \} \leq p_2 \leq 1$, and
\be\label{uuuq}
q\geq\max{ \left\{ \frac{4r_1 L_{*}^2M(1+kb_2w_{*}+b_1B)}{d_1},L_{*}\sqrt{\frac{2}{\sigma_1}}  \right\} },\;  \mbox{ with } \;  M:=\left(\frac{7}{2\sigma_1e}\right)^{7/2}.
\ee
Moreover, $z_0:=-(q/L_{*})^2$ and $z_2\in[-2/\sigma_1,-1/\sigma_1]$ is defined (uniquely) by
\beaa
p_2L_{*}(-z_2)e^{\sigma_1 z_2} = 1.
\eeaa
Also, the choice of $q$ in \eqref{uuuq} implies that $z_0\le -2/\sigma_1$. We shall obtain the following lemma:

\begin{lemma}\label{00vw}
Suppose that $\beta_*>0$ and $s=s_{*}$. {Let \eqref{uur} and \eqref{uuud0} be enforced.}
Then the functions $(\overline{\phi}_1,\overline{\phi}_2,\overline{\phi}_3)$
and $(\underline{\phi}_1,\underline{\phi}_2,\underline{\phi}_3)$ defined in \eqref{uuusols11}-\eqref{uuusols32}
are a pair of generalized upper and lower solutions of \eqref{TWS} in the sense of Definition~\ref{uplodfn}, satisfy \eqref{order}-\eqref{lrd} and are such that boundary condition \eqref{BCu} holds at the unstable tail.
\end{lemma}

The proof of Lemma~\ref{00vw} is also given in Section~\ref{verul}. Then, similarly as before, the first part of Theorem~\ref{th:cs2} is proved, by applying Proposition~\ref{exist} together with Lemmas~\ref{0vw} and~\ref{00vw}

%%%%%%%%%%%%%%%%%%%%%%
%%%%%%%%%%%%%%%%%%%%%%
\section{Asymptotic behavior of stable tail}\label{sec:stable_stail}
\setcounter{equation}{0}

This section is devoted to the proof of the second parts of Theorems~\ref{th:sc1} and \ref{th:cs2}. Throughout this section, we shall denote
\beaa
\phi_j^{+} :=\limsup_{z\rightarrow\infty} \phi_j(z), \; \phi_j^{-} :=\liminf_{z\rightarrow\infty} \phi_j(z),\; j=1,2,3,
\eeaa
where $(\phi_1 ,\phi_2 , \phi_3)$ denotes any of the traveling wave solutions obtained in Subsections~\ref{vv} and~\ref{uu}. We recall that, by construction,
\be\label{ul-bound}
 0 < \phi_1 , \phi_2 (z)\leq 1,\;0 < \phi_3(z)\leq 2a-1,\quad\forall z \in \mathbb{R}.
\ee

\subsection{Preliminaries: the ODE system} Let us first introduce the 6-dimensional first order ODE system corresponding to~\eqref{TWS}:
\be
\begin{cases}\label{TWS_first}
\phi_1 ' = \psi_1, \\
\psi_1 ' = \frac{1}{d_1} [ s \psi_1 - r_1 \phi_1 (1 - \phi_1 - k \phi_2 - b_1 \phi_3 ) ] ,\\
\phi_2 ' = \psi_2, \\
\psi_2 ' = \frac{1}{d_2} [ s \psi_2 - r_2 \phi_2 (1 - h \phi_1 - \phi_2 - b_2 \phi_3 ) ] ,\\
\phi_3 ' = \psi_3, \\
\psi_3 ' = \frac{1}{d_3} [ s \psi_3 - r_3 \phi_3 (-1 + a \phi_1 + a \phi_2 - \phi_3 ) ] .
\end{cases}
\ee
For convenience, we shall denote by $\Psi = (\phi_1,\psi_1,\phi_2,\psi_2,\phi_3,\psi_3)$ any solution of \eqref{TWS_first}, and rewrite \eqref{TWS_first} as $\Psi ' = F (\Psi)$. In this subsection we state several lemmas related to the stable manifolds of various equilibria of~\eqref{TWS_first}. For convenience, we write explicitly its Jacobian matrix:
\newcommand\scalemath[2]{\scalebox{#1}{\mbox{\ensuremath{\displaystyle #2}}}}
\beaa
J_F (\Psi)  = \left(    \scalemath{0.8}{
\begin{array}{cccccc}
0 &1 & 0 & 0 & 0 & 0  \\
 -\frac{r_1}{d_1} (1- 2\phi_1 - k\phi_2 - b_1 \phi_3) &\frac{s}{d_1} &  \frac{r_1}{d_1} k\phi_1& 0 &  \frac{r_1}{d_1} b_1 \phi_1 & 0\\
0 & 0 &0 & 1 & 0 & 0\\
\frac{r_2}{d_2} h\phi_2 & 0 & -\frac{r_2}{d_2} (1-2\phi_2 -h\phi_1-b_2 \phi_3) &\frac{s}{d_2} & \frac{r_2}{d_2} b_2 \phi_2 & 0 \\
0 & 0 & 0 & 0 & 0 & 1 \\
-\frac{r_3}{d_3} a \phi_3 & 0& -\frac{r_3}{d_3} a \phi_3   & 0 &  -\frac{ r_3}{d_3} (-1 + a\phi_1 + a\phi_2 - 2\phi_3) & \frac{s}{d_3}
\end{array}
}
\right)
.
\eeaa
\begin{lemma}\label{ODE_lemma1_ter0}
There exists an open neighborhood $W_0$ of the steady state $(0,0,0,0,0,0)$ such that any solution $\Psi$ of \eqref{TWS_first}  satisfying $\Psi (z) \in W_0$ for all $z\geq 0$
must also satisfy $\Psi (z) \in \{ \phi_1 = 0 , \phi_2 = 0 \}$ for all $z$.
\end{lemma}
\begin{proof}
The matrix of the linearized system of \eqref{TWS_first} around $(0,0,0,0,0,0)$ is
\beaa
J_F (0) =
\left(
\begin{array}{cccccc}
0 &1 & 0 & 0 & 0 & 0  \\
-\frac{r_1}{d_1} &\frac{s}{d_1} & 0 & 0 & 0 & 0\\
0 & 0 &0 & 1 & 0 & 0\\
0 & 0 & -\frac{r_2}{d_2} &\frac{s}{d_2} & 0 & 0 \\
0 & 0 & 0 & 0 & 0 & 1 \\
0  & 0& 0  & 0 &  \frac{ r_3}{d_3}  & \frac{s}{d_3}
\end{array}
\right)
,
\eeaa
which has one negative real eigenvalue and five eigenvalues with positive real parts. By standard perturbation theory, this means that the steady state $(0,0,0,0,0,0)$ of \eqref{TWS_first} has a 1-dimensional stable manifold $\mathcal{S}$.
Furthermore, there is a neighborhood $W_0$ of $(0,0,0,0,0,0)$ such that, for any $\Psi (0) \in W_0 \setminus \mathcal{S}$, there exists $z >0$ such that $\Psi (z) \not \in W_0$.

Now notice that $\{\phi_1 = \psi_1 = 0 , \phi_2 = \psi_2 = 0\}$ is an invariant set for \eqref{TWS_first}. Repeating the same standard stability analysis, one finds that $(0,0,0,0,0,0)$ also admits a 1-dimensional stable manifold in the subset $\{\phi_1 = \psi_1 = 0 , \phi_2 = \psi_2 = 0 \}$. This implies that $\mathcal{S}$ is actually included in $\{\phi_1 = \psi_1 = 0 , \phi_2 = \psi_2 = 0 \}$, and the lemma follows.\end{proof}
The next two lemmas can be proved in the same way.
\begin{lemma}\label{ODE_lemma1_ter2}
There exists an open neighborhood $W_{1,u}$ of the steady state $(1,0,0,0,0,0)$ such that any solution $\Psi$ of \eqref{TWS_first}  satisfying $\Psi (z) \in W_{1,u}$ for all $z\geq 0$
must also satisfy $\Psi (z) \in \{ \phi_2 = 0 , \phi_3 = 0 \}$ for all $z$.

There exists an open neighborhood $W_{1,v}$ of the steady state $(0,0,1,0,0,0)$ such that any solution $\Psi$ of \eqref{TWS_first}  satisfying $\Psi (z) \in W_{1,v}$ for all $z\geq 0$
must also satisfy $\Psi (z) \in \{ \phi_3 = 0 \}$ for all $z$.
\end{lemma}
%\begin{proof}
%For the sake of completeness, we write the jacobian matrices corresponding to these two equilibria. The first one is
%\beaa
%J_F (1,0,0,0,0,0) = \left(
%\begin{array}{cccccc}
%0 &1 & 0 & 0 & 0 & 0  \\
%\frac{r_1}{d_1} &\frac{s}{d_1} &\frac{r_1}{d_1} k & 0 &  \frac{r_1}{d_1} b_1 & 0\\
%0 & 0 &0 & 1 & 0 & 0\\
%0 & 0 & -\frac{r_2}{d_2} (1-h) &\frac{s}{d_2} & 0 & 0 \\
%0 & 0 & 0 & 0 & 0 & 1 \\
%0  & 0& 0  & 0 &  -\frac{ r_3}{d_3} (a -1) & \frac{s}{d_3}
%\end{array}
%\right)
%,
%\eeaa
%which due to $a > 1 > h$ has one negative real eigenvalue (whose corresponding eigenvector belongs to $\R^2 \times \{0\}^4$) and five eigenvalues with positive real part.
%
%The second matrix is
%\beaa
%J_f (0,0,1,0,0,0) =
%\left(
%\begin{array}{cccccc}
%0 &1 & 0 & 0 & 0 & 0  \\
%-\frac{r_1}{d_1} (1-k) &\frac{s}{d_1} &0  & 0 &  0 & 0\\
%0 & 0 &0 & 1 & 0 & 0\\
%\frac{r_2}{d_2} h & 0 & \frac{r_2}{d_2} &\frac{s}{d_2} & \frac{r_2}{d_2} b_2 & 0 \\
%0 & 0 & 0 & 0 & 0 & 1 \\
%0  & 0& 0  & 0 &  -\frac{ r_3}{d_3} (a -1) & \frac{s}{d_3}
%\end{array}
%\right)
%,
%\eeaa
%and because of $k >1$ it has two negative real eigenvalues (whose corresponding eigenvectors belong respectively to $\R^2 \times \{0\}^4$ and $\{0\}^2 \times \R^2 \times \{0\}^2$) and four eigenvalues with positive real part.
%\end{proof}

\begin{lemma}\label{ODE_lemma1_ter1}
Assume that $\beta_* >0$. There exists an open neighborhood $W_*$ of the steady state $(0,0,v_*,0,w_*,0)$ such that any solution $\Psi$ of \eqref{TWS_first}  satisfying $\Psi (z) \in W_*$ for all $z\geq 0$
must also satisfy $\Psi (z) \in \{ \phi_1 = 0 \}$ for all $z$.

Assume that $\beta^* >0$. There exists an open neighborhood $W^*$ of the steady state $(u^*,0,0,0,w^*,0)$ such that any solution $\Psi$ of \eqref{TWS_first}  satisfying $\Psi (z) \in W^*$ for all $z\geq 0$
must also satisfy $\Psi (z) \in \{ \phi_2 = 0 \}$ for all $z$.

\end{lemma}
%\begin{proof}
%After a straightforward computation, one may find that the matrix of the linearized system of \eqref{TWS_first} around $(0,0,v_*,0,w_*,0)$ is
%\beaa
%\left(
%\begin{array}{cccccc}
%0 &1 & 0 & 0 & 0 & 0  \\
%-\frac{r_1}{d_1} \beta_* &\frac{s}{d_1} & 0 & 0 & 0 & 0\\
%0 & 0 &0 & 1 & 0 & 0\\
%\frac{r_2}{d_2} hv_* & 0 & \frac{r_2}{d_2} v_* &\frac{s}{d_2} & \frac{r_2}{d_2} b_2 v_* & 0 \\
%0 & 0 & 0 & 0 & 0 & 1 \\
%-\frac{r_3}{d_3} a w_* & 0& -\frac{r_3}{d_3} a w_*   & 0 &  \frac{ r_3}{d_3} w_* & \frac{s}{d_3}
%\end{array}
%\right)
%,
%\eeaa
%whose characteristic equation is
%\beaa
%\left(\l^2-\frac{s}{d_1}\l+\frac{r_1\beta_*}{d_1}\right)\left[\left(\l^2-\frac{s}{d_2}\l-\frac{r_2v_*}{d_2}\right)\left(\l^2-\frac{s}{d_3}-\frac{r_3w_*}{d_3}\right)+\frac{r_2r_3ab_2v_*w_*}{d_2d_3}\right].
%\eeaa
%which has four eigenvalues with positive real parts and two eigenvalues with negative real parts (some may be complex). Since $\beta_* >0$ and $\{ \phi_1 = 0, \psi_1 = 0 \}$ is an invariant set for \eqref{TWS_first},
%one finds proceeding as in the proof of Lemma~\ref{ODE_lemma1} that the stable manifold of $(0,0,v_*,0,w_*,0)$ is included in $\{ \phi_1 = 0, \psi_1 = 0\}$.
%The first part of the lemma follows, and the second part can be dealt with similarly.
%\end{proof}

\subsection{Some general estimates} For the sake of conciseness, we state here some lemmas that hold in both cases of a strong alien and a weak alien competitor prey. In particular, throughout this subsection $(\phi_1, \phi_2, \phi_3)$ shall still denote any of the traveling wave solution constructed in either Subsections~\ref{vv} or~\ref{uu}.

The first lemmas state that some components of these traveling wave solutions cannot go simultaneously to 0. They rely on a sequential argument inspired by persistence theory in dynamical systems, which has also been used in the context of spreading behavior in predator-prey systems~\cite{DGGS,DGM}.
\begin{lemma}\label{lem:liminf1_ter12}
It holds that
$$\liminf_{z \to \infty} [\phi_1 + \phi_2] (z)  >0.$$
In particular we can define
$$\delta_{12} := \inf_{z \geq 0 } [\phi_1 + \phi_2 ] (z) > 0.$$
\end{lemma}
\begin{proof}
We proceed by contradiction and assume that there exists a sequence $\{z_n\}_{n \in \mathbb{N}}$ such that $z_n \to \infty$ and $\phi_1 (z_n) + \phi_2 (z_n) \to 0$ as $n \to \infty$. Then we let any $\epsilon >0$ arbitrarily small, and we define another sequence
$$z_n ' := \inf \{ z \leq z_n \, | \ \forall z' \in (z, z_n), \ [ \phi_1 + \phi_2]  (z ') \leq \epsilon \},$$
so that $[ \phi_1 + \phi_2 ] (z_n')=\epsilon$ for all $n$. Furthermore, we know from elliptic estimates that $(\phi_1, \phi_2, \phi_3) (\cdot + z_n)$ converges up to extraction of a subsequence to a solution $(\bar{\phi}_1, \bar{\phi}_2, \bar{\phi}_3)$ of~\eqref{TWS}, and that $\bar{\phi}_1 \equiv \bar{\phi}_2 \equiv 0$ by the strong maximum principle. Hence $z_n - z_n '\to \infty$ as $n \to \infty$.

Passing to the limit as $n \to \infty$, we find that $(\phi_1, \phi_2, \phi_3) (\cdot + z'_n)$ converges to another solution of \eqref{TWS},
which we denote by $(\hat{\phi}_1^\epsilon, \hat{\phi}_2^\epsilon, \hat{\phi}_3^\epsilon)$. Furthermore, by construction we have that $\hat{\phi}_1^\e  (0) + \hat{\phi}_2^\e (0) = \epsilon$ and $0 \leq \hat{\phi}_1^\e , \hat{\phi}_2^\e \leq \epsilon$ for all $z \geq 0$.

We now claim that there exists $\delta (\epsilon) \to 0$ as $\epsilon \to 0$ such that
\be\label{claim0_ter12}
 | \hat{\phi}_3^\e (z) | \leq \delta (\epsilon),\; \forall\, z \geq 0.
\ee
We proceed by contradiction and assume that there exist sequences $\e_n \to 0$ and $y_n \geq 0$
such that $\lim_{n \to \infty} \hat{\phi}_3^{\e_n} (y_n) > 0$. By standard elliptic estimates, we find that $(\hat{\phi}_1^{\e_n}, \hat{\phi}_2^{\e_n}) (y_n + z) \to (0,0)$, and then that $\hat{\phi}_3^{\e_n}  (y_n  + x + st)$ converges as $n \to \infty$ to an entire in time solution $w(x,t)$ of
$$w_t = d_3 w_{xx} + r_3 w (-1 -w),$$
which is also bounded from above by $2 a -1$. Thus this limit must be identical to 0, a contradiction. Claim~\eqref{claim0_ter12} is proved.

Next, as $\e \to 0$, we get that $(\hat{\phi}_1^\e , \hat{\phi}_2^\e , \hat{\phi}_3^\e) \to (0,0,0)$ uniformly in $[0,+\infty)$. By standard elliptic estimates, it also follows that $((\hat{\phi}_1^\e )', (\hat{\phi}_2^\e)' , (\hat{\phi}_3^\e)')  \to (0,0,0)$ in $[0,+\infty)$. Therefore, we can find $\e$ small enough so that
$$(\hat{\phi}_1^\e , (\hat{\phi}_1^\e )', \hat{\phi}_2^\e , (\hat{\phi}_2^\e)' , \hat{\phi}_3^\e , (\hat{\phi}_3^\e)') \in W_0,$$
for all $z \geq 0$. Applying Lemma~\ref{ODE_lemma1_ter0}, we infer that $\hat{\phi}_1^\e \equiv \hat{\phi}_2^\e \equiv 0$.
However, by construction we have that $\hat{\phi}_1^\e(0) + \hat{\phi}_2^\e(0) = \varepsilon >0$, thus we have reached a contradiction. The lemma is proved.
\end{proof}

\begin{lemma}\label{lem:liminf1_ter23}
It holds that
$$\liminf_{z \to \infty} [\phi_2 + \phi_3] (z) >0.$$
In particular we can define
$$\delta_{23} := \inf_{z \geq 0 } [\phi_2 + \phi_3 ] (z) > 0.$$
\end{lemma}
\begin{proof}
We again proceed by contradiction and assume that there exists a sequence $\{z_n\}_{n \in \mathbb{N}}$ such that $z_n \to \infty$ and $\phi_2 (z_n) + \phi_3 (z_n) \to 0$ as $n \to \infty$. Then we let any $\epsilon >0$ arbitrarily small, and we define another sequence
$$z_n ' := \inf \{ z \leq z_n \, | \ \forall z' \in (z, z_n), \ [ \phi_2 + \phi_3]  (z ') \leq \epsilon \},$$
so that $[ \phi_2 + \phi_3 ] (z_n')=\epsilon$ for all $n$. As in the previous lemma, it follows from a limit argument and a strong maximum principle that $z_n - z_n '\to \infty$ as $n \to \infty$.

Passing to the limit as $n \to \infty$, we find that $(\phi_1, \phi_2, \phi_3) (\cdot + z'_n)$ converges to another solution of \eqref{TWS},
which we denote by $(\hat{\phi}_1^\epsilon, \hat{\phi}_2^\epsilon, \hat{\phi}_3^\epsilon)$. Furthermore, by construction we have that $\hat{\phi}_2^\e  (0) + \hat{\phi}_3^\e (0) = \epsilon$ and $0 \leq \hat{\phi}_2^\e , \hat{\phi}_3^\e \leq \epsilon$ for all $z \geq 0$. Provided that $\varepsilon >0 $ is small, we also have by Lemma~\ref{lem:liminf1_ter12} that
\be\label{claim_detail}
\hat{\phi}_1^\e (z) \geq \frac{\delta_{12}}{2} >0,
\ee
for any $z \geq 0$.

We now claim that there exist $\delta (\epsilon) \to 0$ as $\epsilon \to 0$ and $z_\epsilon$ such that
\be\label{claim0_ter13}
 | 1 -\hat{\phi}_1^\e (z) | \leq \delta (\epsilon),\; \forall\, z \geq z_\epsilon .
\ee
Indeed, we proceed by contradiction and assume that there exist sequences $\e_n \to 0$ and $y_n \to \infty$
such that $\lim_{n \to \infty} \hat{\phi}_1^{\e_n} (y_n) < 1$. By standard elliptic estimates, we find that $(\hat{\phi}_2^{\e_n}, \hat{\phi}_3^{\e_n}) (y_n + z) \to (0,0)$, and then that $\hat{\phi}_1^{\e_n}  (y_n  + x + st)$ converges as $n \to \infty$ to an entire in time solution $u(x,t)$ of
$$u_t = d_1 u_{xx} + r_1 u (1 -u ),$$
which is also bounded from below by $\frac{\delta_{12}}{2}$ (recall \eqref{claim_detail} and that $y_n \to \infty$).  Thus this limit must be identical to 1. We have reached a contradiction and the claim~\eqref{claim0_ter13} is proved.

Therefore, for $\e$ small enough we have found a solution of \eqref{TWS} which is in $W_{1,u}$ for all $z \geq z_\epsilon$,
hence $\hat{\phi}_2^\e \equiv \hat{\phi}_3^\e  \equiv 0$ by Lemma~\ref{ODE_lemma1_ter2} (with a shift $z\rightarrow z-z_\e$).
This contradicts the fact that $\hat{\phi}_2^\e (0) + \hat{\phi}_3^\e (0) = \e >0$. The lemma is proved.
\end{proof}

It follows that, in all cases, the predator cannot go to extinction at the stable tail.

\begin{lemma}\label{lem:liminf1_ter3}
It holds that $\phi_3^- = \liminf_{z \to \infty} \phi_3 (z) >0$.
\end{lemma}
\begin{proof}
The argument is again very similar. As in the proofs of Lemmas~\ref{lem:liminf1_ter12} and~\ref{lem:liminf1_ter23}, we assume by contradiction that, for any small $\varepsilon >0$, there exist sequences $\{z_n\}_{n \in \mathbb{N}}$ and $\{z_n ' \}_{n \in \mathbb{N} }$ such that $z_n - z_n ' \to \infty$ as $n \to \infty$, and
$$\phi_3 (z_n ' ) = \varepsilon \geq \phi_3 (z),$$
for all $z \in [z_n ' , z_n]$. From elliptic estimates, $(\phi_1, \phi_2, \phi_3) (\cdot + z'_n)$ converges to another solution of \eqref{TWS}, which we denote by $(\hat{\phi}_1^\epsilon, \hat{\phi}_2^\epsilon, \hat{\phi}_3^\epsilon)$. Furthermore, we have that $\hat{\phi}_3^\e (0) = \epsilon$ and $ \hat{\phi}_3^\e \leq \epsilon$ for all $z \geq 0$. Up to reducing $\varepsilon >0 $, we also have by Lemma~\ref{lem:liminf1_ter23} that
$$\hat{\phi}_2^\e (z) \geq \frac{\delta_{23}}{2} >0$$
for any $z \geq 0$.

Next we claim that there exist $\delta (\epsilon) \to 0$ as $\epsilon \to 0$ and $z_\epsilon$ such that
\be\label{claim0_ter3}
|\hat{\phi}_1^\e (z) | +   | 1 -\hat{\phi}_2^\e (z) | \leq \delta (\epsilon),\; \forall\, z \geq z_\epsilon.
\ee
Indeed, taking any sequence $z_n \to \infty$ as $n \to \infty$, we have up to extraction of a subsequence that the pair $(\hat{\phi}_1^\e, \hat{\phi}_2^\e)(\cdot+ z_n)$ converges to $(\tilde{\phi}_1, \tilde{\phi}_2)$ which satisfies
$$0 \leq \tilde{\phi}_1 \leq 1, \; \frac{\delta_{23}}{2} \leq \tilde{\phi}_2 \leq 1,$$
and
\be\label{tilde12}
\begin{cases}
  d_1\tilde{\phi}_1''(z) -s\tilde{\phi}_1'(z)+ r_1\tilde{\phi}_1(z)[ 1-\tilde{\phi}_1(z)-k\tilde{\phi}_2(z)] \geq 0,\;z\in\R,\\
  d_2\tilde{\phi}_2''(z) -s\tilde{\phi}_2'(z)+ r_2 \tilde{\phi}_2(z)[ 1-h \tilde{\phi}_1(z)-\tilde{\phi}_2(z) - b_2 \epsilon ] \leq 0,\;z\in\R.
\end{cases}
\ee
Letting $\overline{\psi}_1  = 1- \tilde{\phi}_1$ and $\overline{\psi}_2 = \tilde{\phi}_2$, we get that $(\overline{\psi}_1, \overline{\psi}_2)$ is a nonnegative supersolution of
\be\label{tilde12bis}
\begin{cases}
    \partial_t \psi_1 -  d_1 (\psi_1)_{zz} +  s (\psi_1)_z -  r_1 (1- \psi_1)(-\psi_1 + k \psi_2 )= 0, \;z\in\R,~t>0 ,\\
    \partial_t \psi_2 - d_2 (\psi_2)_{zz} + s (\psi_2)_z  - r_2 \psi_2 ( 1-h +h \psi_1- \psi_2  - b_2 \epsilon ) = 0, \;z\in\R,~t>0,
\end{cases}
\ee
which is a cooperative reaction-diffusion system and hence satisfies a comparison principle.
In particular, we must have that $\overline{\psi}_1 (z) \geq \underline{\psi}_1 (t)$ and $\overline{\psi}_2 (z) \geq \underline{\psi}_2 (t)$ for all $t>0$ and $z \in \R$, where $(\underline{\psi}_1, \underline{\psi}_2)$ solves \eqref{tilde12bis} with the initial data $(0,\delta_{23}/2)$; notice that, due to the invariance by translation, $(\underline{\psi}_1, \underline{\psi}_2)$ does not depend on the spatial variable $z$.

Furthermore, provided that $\epsilon$ is small enough, then $(0,\delta_{23}/2)$ and $(1,1)$ are respectively a sub and a supersolution of \eqref{tilde12bis}. By the comparison principle and parabolic estimates, it follows that $\underline{\psi}_1$ and $\underline{\psi}_2$ are nondecreasing in time and converge to a constant steady state $(P,Q)$ of \eqref{tilde12bis}, with $0 \leq P \leq 1$ and $\delta_{23}/2 \leq Q \leq 1$. It is straightforward to check that, when $\epsilon >0$ is small enough, the only such steady state is $(1,1-b_2 \epsilon)$.
Putting the above facts together, we find that $\tilde{\phi}_1\equiv 0$ and $\tilde{\phi}_2\in[1 - b_2 \epsilon,1]$. Hence \eqref{claim0_ter3} is proved.

Then we get for $\e$ small enough a solution of \eqref{TWS} which is in $W_{1,v}$ for all $z \geq z_\epsilon$, hence $ \hat{\phi}_3^\e  \equiv 0$ by Lemma~\ref{ODE_lemma1_ter2}.
This contradicts our construction and the lemma is proved.
\end{proof}

We complete this subsection with a lemma which shall allow us to derive the stable tail limit, regardless of the alien species, in the co-existence case.

\begin{lemma}\label{LE-entire-full}
Assume that $E_c$ exists, i.e., \eqref{co-ex} holds, and if also
\be\label{L-extra}  k \sqrt{\frac{b_2}{b_1}} + h \sqrt{\frac{b_1}{b_2}}  < 2,
\ee
let $(u,v,w) = (u,v,w) (x,t)$ be a bounded entire solution of \eqref{main} such that
\begin{equation}\label{lb}
m:= \min\left(\inf_{(x,t)\in\R^2} u(x,t), \inf_{(x,t)\in\R^2}v(x,t),\inf_{(x,t)\in\R^2}w(x,t)\right) >0.
\end{equation}
Then
$
(u,v,w) \equiv \left(u_c  , v_c , w_c \right)$.
\end{lemma}

\begin{proof}
First, we denote $M = \max \{ \|u\|_\infty, \|v\|_\infty , \| w \|_\infty \}$ and define the functions $g (x)= x - \ln (x) -1$ and $\Phi=\Phi(u,v,w)$ given by
$$
\Phi(u,v,w):= \frac{ r_3 a u_c}{b_1 r_1} g\left(\frac{u}{u_c }\right)+\frac{r_3 a v_c }{b_2 r_2 }g\left(\frac{v}{v_c }\right)+ w_c g\left(\frac{w}{w_c }\right).
$$
Let us compute the Lie derivative of $\Phi$, denoted by $L_X \Phi$, along the three dimensional vector field
$$ X:=(r_1u (1 -u -kv - b_1 w, r_2v (1-hu - v - b_2 w), r_3w (-1 + au + av - w) )$$
associated to the kinetic part of \eqref{main}. Then, using that $(u_c, v_c ,w_c)$ is a stationary state, we find
\beaa
L_X\Phi( u,v,w) &= & (\Phi_u,\Phi_v,\Phi_w)\cdot X\\
& = & \frac{r_3 a}{b_1 } (u-u_c) (1-u - kv - b_1 w) + \frac{r_3 a}{b_2} (v- v_c) (1- v - h u - b_2 w) \\
& & + r_3 (w - w_c) (-1 + a u + av - w) \\
& = & - \frac{r_3 a}{b_1 } (u-u_c)^2 -  \frac{r_3 a}{b_2} (v- v_c)^2 - r_3 (w - w_c)^2 \\
& & - \left( \frac{r_3 a k}{b_1} + \frac{r_3 a h}{b_2} \right) (u-u_c) (v-v_c) .
\eeaa
On the other hand, for any $(X_1, X_2) \in \R^2$, we have
\beaa
X_1^2 + X_2^2  - \left( \frac{k}{b_1} + \frac{h}{b_2} \right) \sqrt{b_1 b_2} |X_1| |X_2| \geq \left[ 1 - \frac{1}{2}\left( \frac{k}{b_1} + \frac{h}{b_2} \right) \sqrt{b_1 b_2}\right] (X_1^2+X_2^2).
\eeaa
Taking $X_1 = \sqrt{\frac{r_3 a}{b_1}} (u-u_c)$ and $X_2 = \sqrt{\frac{r_3 a}{b_2}} (v-v_c)$, we infer from \eqref{L-extra} that
$$L_X \Phi (u,v,w) \leq - \alpha \left[ (u-u_c)^2 + (v-v_c)^2 + (w-w_c)^2 \right]$$
for some $\alpha >0$ and any $u,v,w>0$.  Furthermore, recalling the definition of $\Phi$ above, there exists $\beta>0$ such that
 \begin{equation*}
 L_X\Phi( u,v,w)\leq -\beta \Phi( u,v,w),\;\forall (u,v,w)\in [m,M]^3 .
 \end{equation*}
From this inequality, the proof of Lemma \ref{LE-entire-full} follows from the same arguments as {that in \cite[Lemma 4.1]{DGGS}.}
\end{proof}
We point out that similar results hold for the two species predator-prey system:
\begin{lemma}\label{prelim_predatorprey}
Any entire in time solution of
\be\label{2_predatorprey}
\begin{cases}
u_t = d_1 u_{xx} + r_1 u (1- u -b_1 w), \\
w_t = d_3 w_{xx} + r_3w(-1+a u -w),
\end{cases}
\ee
such that
$$0 < \min \{ \inf_{\R^2} u , \inf_{\R^2} w \} < \max \{ \sup_{\R^2} u , \sup_{\R^2} w \} < +\infty,$$
must satisfy that $u \equiv u^*$ and $w \equiv w^*$.

Similarly, any entire in time solution of the subsystem derived from~\eqref{main} by letting $u=0$, and satisfying
$$0 < \min \{ \inf_{\R^2} v, \inf_{\R^2} w \} < \max \{ \sup_{\R^2} v , \sup_{\R^2} w \} < +\infty,$$
must satisfy that $v \equiv v_*$ and $w \equiv w_*$.
\end{lemma}
\begin{proof}
The proof of this lemma can be found in \cite[Lemma 4.2]{DGGS} by a Lyapunov argument (see also \cite{DH04,DG18}). We omit it here.
\end{proof}

%%%%%%%%%%%%%%
\subsection{The case of the alien strong competitor}

Now we turn to the case when the strong competing prey is the alien species, and in particular we assume here that $\b^*>0$, i.e., \eqref{positive} holds.
Now $(\phi_1, \phi_2, \phi_3)$ denotes the traveling wave solution constructed in Subsection~\ref{vv}.
In order to apply either method of contracting rectangles (see~\cite{GNOW20}) or Lyapunov function (see Lemma~\ref{LE-entire-full} above), positive lower bounds on the traveling wave solutions are typically required.

We already know by Lemma~\ref{lem:liminf1_ter3} that $\phi_3^- >0$. We therefore continue the proof of the stable tail limit by obtaining some better lower bounds for $\phi_2$.
Because $\phi_1$ only persists when the invading state is $E_c$, it shall be considered separately in Subsection~\ref{sec:coexist_limit}.

\begin{lemma}\label{lem:liminf1_ter2}
It holds that $\phi_2^- = \liminf_{z \to \infty} \phi_2 (z) >0$.
\end{lemma}
\begin{proof}
We again proceed by contradiction and, as in the proof of Lemma~\ref{lem:liminf1_ter3}, for any small $\varepsilon >0$
we find sequences $\{z_n\}_{n \in \mathbb{N}}$ and $\{ z_n ' \}_{n \in \mathbb{N}}$ such that $z_n - z_n '\to \infty$, $\phi_2 (z_n')=\epsilon$ and $\phi_2  \leq \varepsilon$ in $[z_n ', z_n]$.

By standard elliptic estimates, we find that $(\phi_1, \phi_2, \phi_3) (\cdot + z'_n)$ converges to another solution of \eqref{TWS},
which we denote by $(\hat{\phi}_1^\epsilon, \hat{\phi}_2^\epsilon, \hat{\phi}_3^\epsilon)$. Furthermore, by construction we have that $\hat{\phi}_2^\e (0) = \epsilon$ and $\hat{\phi}_2^\e \leq \epsilon$ for all $z \geq 0$.
Also, due to Lemmas~\ref{lem:liminf1_ter12} and~\ref{lem:liminf1_ter23},
there exist some $\delta_0 >0$ independent of $\e$ such that $\hat{\phi}_1^\e(z)$, $\hat{\phi}_3^\e(z) \geq \delta_0$ for all $z \geq 0$.

We then claim that there exist $z_\e$ and $\delta (\epsilon) \to 0$ as $\epsilon \to 0$ such that
\be\label{claim0_ter}
|\hat{\phi}_1^\e (z )  - u^* | + | \hat{\phi}_3^\e (z) - w^*| \leq \delta (\epsilon),\; \forall\, z \geq z_\e.
\ee
Indeed, proceed by contradiction and assume that there exists $z^\epsilon \to \infty $ such that for instance $\liminf | \hat{\phi}_1^\e ( z^\epsilon ) - u^*| >0$ when (some subsequence of) $\epsilon \to 0$.
By parabolic estimates and up to extraction of another subsequence, we get that $(\hat{\phi}_1^\epsilon, \hat{\phi}_2^\epsilon, \hat{\phi}_3^\epsilon) (\cdot + z^\e)$ converges to a solution $(\hat{\phi}_1^0, \hat{\phi}_2^0 , \hat{\phi}_3^0)$ of \eqref{TWS}. Moreover, we have by construction that $\hat{\phi}_2^0 \equiv 0$, and therefore $(\hat{\phi}_1^0 , \hat{\phi}_3^0) (x + st)$ is an entire in time solution of~\eqref{2_predatorprey}. It also follows from our construction that $0 <\delta_0 \leq \hat{\phi}_1^0 \leq 1$ and $0 < \delta_0 \leq \hat{\phi}_3^0 \leq 2a -1$, hence from Lemma~\ref{prelim_predatorprey} that $\hat{\phi}_1^0 \equiv u^*$ and $\hat{\phi}_3^0 \equiv w^*$. We have reached a contradiction and proved the claim~\eqref{claim0_ter}.

Therefore, we can choose $\epsilon>0$ small enough so that $(\hat{\phi}_1^\epsilon,( \hat{\phi}_1^\epsilon)  ' , \hat{\phi}_2^\epsilon , (\hat{\phi}_2^\epsilon) ', \hat{\phi}_3^\epsilon,  (\hat{\phi}_3 ^\epsilon )') \in  W^*$, for all $z \geq z_\e$. Since $\beta^* >0$ and applying Lemma~\ref{ODE_lemma1_ter1}, we conclude that $\hat{\phi}_2^\epsilon \equiv 0$. This contradicts the fact that $\hat{\phi}_2^\epsilon (0 ) =\epsilon >0$, and ends the proof of the lemma.
\end{proof}

%%%%%%%%%%%%%%%%%%%%%
\subsubsection{Semi-co-existence case}

This subsection is devoted to showing that the traveling wave solution obtained in Subsection~\ref{vv} approaches $E_{*} = (0,v_* , w_*)$ as $z\to\infty$, if $\b_*<0$ and~\eqref{hb2} holds.
%{Although the idea is similar to that in \cite{GNOW20}, there is a major difference which is described below.}

Recall that
\beaa
\phi_j^{+} :=\limsup_{z\rightarrow\infty} \phi_j(z), \; \phi_j^{-} :=\liminf_{z\rightarrow\infty} \phi_j(z),\; j=1,2,3,
\eeaa
Unfortunately, the implicit lower bound in Lemma~\ref{lem:liminf1_ter2} on $\phi_2^-$ is not enough for our purpose. We immediately improve it:
\begin{lemma}\label{p2-lower}
It holds that $\phi_2^- \ge 1-h-b_2(2a-1):=\gamma_2$.
\end{lemma}
\begin{proof}
The proof is the same as that of Lemma 4.1 in \cite{GNOW20}.
Take any sequence $z_n \to \infty$, and by elliptic estimates assume up to extraction of a subsequence that $\phi_2 (\cdot + z_n) \to \tilde{\phi}_2$ as $n \to \infty$.
Then $\tilde{\phi}_2 \geq  \phi_2^-  >0$ and, by \eqref{ul-bound}, it satisfies
$$d_2 \tilde{\phi}_2 '' - s \tilde{\phi}_2 ' + r_2 \tilde{\phi}_2 [ 1 - h  - \tilde{\phi}_2 - b_2 (2a -1) ] \leq 0.$$
On the other hand, the solution of the ODE
$$\partial_t \underline{u} = r_2 \underline{u} [ 1-h - \underline{u} - b_2 (2a -1) ],$$
with initial condition
$$\underline{u} (t=0) = \phi_2^- >0 ,$$
converges to $\gamma_2$ as $t \to +\infty$. Since $\tilde{\phi}_2 \geq \phi_2^-$, we can apply the parabolic comparison principle on the whole real line, and we find that $\tilde{\phi}_2 \geq \gamma_2$. Due to the arbitrary choice of the sequence~$z_n$, we conclude as wanted that $\phi_2^- = \gamma_2$.
\end{proof}

Now define
\beaa
\begin{cases}
 %m_1(\ta) := (1-\ta)(-\e^2)                  , \quad M_1(\ta) := (1-\ta)(1+\e^2),\\
 m_2(\ta) := (1-\ta)(\gamma_2-\e) + \ta v_{*}, \quad M_2(\ta) := (1-\ta)(1+\e^2) + \ta v_{*},\\
 m_3(\ta) := (1-\ta)(\delta_3 -\e) + \ta w_{*}, \quad M_3(\ta) := (1-\ta)(2a-1+\e)+ \ta w_{*},
\end{cases}
\eeaa
where $\gamma_2= 1 - h - b_2 (2a -1) >0$, $\delta_3 := \min \left\{ {w_*}/{2}, {(a\gamma_2 -1)}/{2} , \phi_3^-  \right\}>0$, and $\e$ satisfies
\be\label{seps}
0<\e<\min{ \left\{\gamma_2, \delta_3, \frac{hk\gamma_2+hb_1\delta_3}{hk+hb_1+b_2}, \frac{ak\gamma_2+ab_1\delta_3}{ak+ab_1+1}, \frac{a\gamma_2- \delta_3 - 1}{a}\right\} }.
\ee
Notice that $a \gamma_2 -1 >0$ follows from~\eqref{hb2}, which in turn ensures together with Lemma~\ref{lem:liminf1_ter3} that such constants~$\delta_3$ and~$\varepsilon$ indeed exist.
Also, due to $a>1$ and the definition of $\delta_3$, we have that $0<\gamma_2<v_*<1$ and $\delta_3<w_*<2a-1$.
Hence $m_j(\theta)$ is increasing and $M_j(\theta)$ is decreasing in $\theta\in[0,1]$ for $j=2,3$.

Due to the lack of a positive lower bound for $\phi_1$, instead of considering 3-d rectangles, we consider the following 2-d contracting rectangles:
\be\label{recQ}
Q(\theta):=[m_2(\theta),M_2(\theta)]\times [m_3(\theta),M_3(\theta)]\subset (0,\infty)^2, \quad \theta\in[0,1].
\ee
Also, we consider the set
\be\label{setA}
\mathcal{A}:=\{\theta\in[0,1)\mid m_k(\ta)<\phi_k^{-}\leq\phi_k^{+}<M_k(\ta),\, k=2,3\}.
\ee
Obviously, by \eqref{ul-bound} and Lemma~\ref{p2-lower}, we have
\beaa
%&& m_1(0)=-\e^2       <0       \leq \phi_1^{-}\leq \phi_1^{+}\leq 1<1+\e^2     = M_1(0),\\
&& m_2(0)=\gamma_2-\e <\gamma_2\leq \phi_2^{-}\leq \phi_2^{+}\leq 1<1+\e^2     = M_2(0),\\
&& m_3(0)=\delta_3-\e <\delta_3\leq \phi_3^{-}\leq \phi_3^{+}\leq 2a-1<2a-1+\e = M_3(0).
\eeaa
Hence $0\in\mathcal{A}$ and also the quantity $\ta_0 :=\sup\mathcal{A}$ is well-defined such that $\ta_0\in(0,1]$.

By passing to the limit, we have
\be\label{p23-bound}
m_j(\ta_0)\le\phi_j^-\le\phi_j^+\le M_j(\ta_0),\quad j=2,3.
\ee
To proceed further, we derive a better upper bound for $\phi_1$ as follows.

\begin{lemma}\label{la:p1+}
Under the condition \eqref{p23-bound}, it holds that $$\phi_1^+\le M_1(\ta_0) := \max \{0, 1-k m_2(\ta_0)- b_1 m_3(\ta_0)\}.$$
\end{lemma}

\begin{proof}
Taking any sequence $\{z_n\}_{n \in \mathbb{N}}$ tending to $\infty$, up to extraction of a subsequence, we have $\phi_1(\cdot+z_n)\to\hat{\phi}_1$ as $n\to\infty$. It follows from \eqref{p23-bound} that
\beaa
d_1\hat{\phi}_1''-s\hat{\phi}_1'+r_1\hat{\phi}_1[1-\hat{\phi}_1-km_2(\theta_0)-b_1m_3(\ta_0)]\ge 0,
\eeaa
on the whole real line. Recalling also that $\phi_1 \leq 1$ and applying the parabolic comparison principle, we get that $\hat{\phi}_1 (z) \le \overline{u} (t)$ for any $z \in \mathbb{R}$ and $t>0$, where $\overline{u}$ solves
$$\partial_t \overline{u} = r_1 \overline{u} [ 1- \overline{u} - km_2 (\theta_0 ) - b_1 m_3 (\theta_0)],$$
with initial condition
$$\overline{u} (t=0) = 1.$$
One may check that $\overline{u} (t) \to M_1 (\theta_0)$ as $t \to +\infty$, hence $\hat{\phi}_1 \leq M_1 (\theta_0)$. Due to the arbitrary choice of the sequence~$z_n$, we reach the wanted conclusion.
\end{proof}

Next, we prove that $\ta_0=1$. We assume by contradiction that $\ta_0\in(0,1)$. In particular, one of the following equalities must hold:
\be\label{equal}
\phi_j^-=m_j(\ta_0),\; \phi_j^+=M_j(\ta_0),\; j=2,3.
\ee
To reach a contradiction with \eqref{equal}, we introduce and compute
\beaa
\alpha_2 & := & 1-hM_1(\ta_0)-m_2(\ta_0)-b_2M_3(\ta_0) ,\\
\omega_2 &:=&  1-M_2(\ta_0)-b_2m_3(\ta_0)  = -(1-\ta_0)[\e^2+b_2(\d_3-\e)]<0,\\
\alpha_3 &:=& -1 +am_2(\ta_0) -m_3(\ta_0) =(1-\ta_0)[a\gamma_2-1-\delta_3-\e(a-1)]>0,\\
\omega_3 &:=& -1 +aM_1(\ta_0) +aM_2(\ta_0) -M_3(\ta_0).
\eeaa
We also compute that
$$\alpha_2 > 0,$$
and to do so we distinguish the two cases when $M_1 (\theta_0) >0$ or $M_1 (\theta_0) =0$. In the former, we have
\beaa
%&& \omega_1:=  1-M_1(\ta_0)-km_2(\ta_0)-b_1m_3(\ta_0)  < (1-\ta_0)[-\e^2-k(\d_2-\e)-b_1(\d_3-\e) ]<0,\\
\alpha_2 %&:=&  1-hM_1(\ta_0)-m_2(\ta_0)-b_2M_3(\ta_0) \\
& =& (1-h)+(1-\ta_0)[-(1-hk)\gamma_2+hb_1\delta_3-b_2(2a-1)]\\
&&+\ta_0[(hk-1)v_*+hb_1w_*-b_2w_*]-(1-\ta_0)\e[-(1-hk)+hb_1+b_2]\\
&=&( 1-h) + (1-\ta_0)[-(1-hk)\gamma_2+hb_1\delta_3-b_2(2a-1)]\\
&&+\ta_0[h(1-\beta_*)-1]-(1-\ta_0)\e[-(1-hk)+hb_1+b_2]\\
&=&(1-\ta_0)\{(hk\gamma_2+hb_1\delta_3)-\e(-1+hk+hb_1+b_2)\}-\ta_0h\beta_*,
\eeaa
using $\gamma_2 = 1-h - b_2 (2a -1)$, $v_*+b_2w_*=1$ and $\beta_*=1-kv_*-b_1w_*$.
Since $\beta_* < 0$, it easily follows from \eqref{seps} that $\alpha_2 >0$ in that case.
In the other case when $M_1 (\theta_0) = 0$, we have
\beaa
\alpha_2 = (1-\theta)[h+(1-b_2)\e] >0.
\eeaa
Similarly, when $M_1(\ta_0)>0$, we have
\beaa
\omega_3 =  -(1-\ta_0)[(ak\gamma_2+ab_1\delta_3)-\e(ak+ab_1+1-a\e)]+a\ta_0\beta_*<0,
\eeaa
using $\beta_*<0$ and \eqref{seps}. On the other hand, $\omega_3=-(1-\ta_0)(a+\e-a\e^2)<0$ when $M_1(\ta_0)=0$.

{From these inequalities, we can get a contradiction following an argument given in \cite{GNOW20}. For instance, if $\phi_2^- =m_2 (\ta_0)$ in \eqref{equal}, then there exists a sequence $z_n \to \infty$ such that $(\phi_1, \phi_2, \phi_3) (\cdot + z_n)$ converges to a solution $(\phi_{1,\infty}, \phi_{2,\infty}, \phi_{3,\infty})$ of~\eqref{TWS} such that
\beaa
&&0 \leq \phi_{1,\infty} \leq M_1 (\ta_0), \quad m_3 (\ta_0) \leq \phi_{3, \infty} \leq M_3 (\ta_0),\\
&&\phi_{2, \infty} \geq \phi_{2,\infty} (0) = m_2 (\ta_0) >0.
\eeaa
Evaluating the equation for $\phi_{2,\infty}$ at 0, one finds a contradiction with the fact that $\alpha_2 >0$. Other cases can be dealt with similarly.}

Hence $\theta_0=1$. This implies that $\phi_2^-=\phi_2^+=v_*$ and $\phi_3^-=\phi_3^+=w_*$, i.e.,
\be\label{p23}
\lim_{z\to\infty}\phi_2(z)=v_*,\quad \lim_{z\to\infty}\phi_3(z)=w_*.
\ee
Finally, applying again Lemma~\ref{la:p1+} with $\theta_0 =1$ and recalling that $1 - k v_* - b_1 w_* < 0$, we also conclude that $\phi_1^- = \phi_1^+ = 0$. The proof is now completed.

%
%Finally, we claim that $\phi_1^-=\phi_1^+=0$, which we prove by contradiction as follows. Notice that $0\le\phi_1^-\le\phi_1^+$, so suppose that $\phi_1^+>0$. We briefly consider two cases.
%
%{\bf Case 1.} $\phi_1(\infty)=\phi_1^+$ exists. Then, by \eqref{p23},
%\beaa
%\lim_{z\to\infty}\{1-\phi_1(z)-k\phi_2(z)-b_1\phi_3(z)\}=1-\phi_1^+ -kv_* -b_1w_*=\b_*-\phi_1^+<0,
%\eeaa
%due to $\b_*<0$.
%Taking $z_0$ large enough such that
%\beaa
%1-\phi_1(z)-k\phi_2(z)-b_1\phi_3(z)\le (\b_*-\phi_1^+)/2,\;\forall\, z\ge z_0,
%\eeaa
%and integrating the first equation in \eqref{TWS} from $z_0$ to $\infty$, we reach a contradiction.
%
%{\bf Case 2.} $\phi_1^-<\phi_1^+$. In this case, we take a maximal sequence $(z_n)_{n \in \mathbb{N}}$ of $\phi_1$, i.e., such that $z_n\to\infty$ and $\phi_1(z_n)\to\phi_1^+$ as $n\to\infty$.
%Then $d_1\phi_1''(z_n)-s\phi_1'(z_n)\le 0$ for all $n$. But,
%\beaa
%\lim_{n\to\infty}\{r_1\phi_1(z_n)[1-\phi_1(z_n)-k\phi_2(z_n)-b_1\phi_3(z_n)]\}=r_1\phi_1^+(1-\phi_1^+-kv_*-b_1w_*)<0,
%\eeaa
%which is impossible.
%
%Thus we conclude that $\phi^+=0$ and so $(\phi_1,\phi_2,\phi_3)(\infty)=(0,v_*,w_*)$. The proof is now completed.}
%

%%%%%%%%%%%%%%%
\subsubsection{Co-existence case}\label{sec:coexist_limit}

In this subsection, we show that the traveling wave solutions obtained in Subsection~\ref{vv} converge to $E_c$ as $z\to\infty$, if \eqref{co-ex} and \eqref{ODE_lyapu} hold. Recall that in particular $\beta_* >0$; see the discussion in Subsection~\ref{sec:ODE}.

Since the method of contracting rectangles is not applicable in this case, we switch to a Lyapunov argument. However, in applying this method, we still need to derive some positive lower bounds for all components. The positive lower bound of $\phi_1$ is derived by a similar argument as that in Lemma~\ref{lem:liminf1_ter2}.

\begin{lemma}\label{lem:liminf1_bis}
It holds that $\phi_1^- = \liminf_{z \to \infty} \phi_1 (z) >0$.
\end{lemma}
\begin{proof}
By contradiction, we find for any small $\varepsilon >0$ a solution $(\hat{\phi}_1^\epsilon, \hat{\phi}_2^\epsilon, \hat{\phi}_3^\epsilon)$ of \eqref{TWS}, such that that $\hat{\phi}_1^\e (0) = \epsilon$ and $\hat{\phi}_1^\e \leq \epsilon$ for all $z \geq 0$. Furthermore, by Lemmas~\ref{lem:liminf1_ter3} and~\ref{lem:liminf1_ter2}, there exists $\delta >0$ such that $\hat{\phi}_2^\e , \hat{\phi}_3^\e \geq \delta$ on $[0,+\infty)$.

As in the proof of Lemma~\ref{lem:liminf1_ter2} and more specifically of~\eqref{claim0_ter}, one can then use Lemma~\ref{prelim_predatorprey} to infer that
\beaa
|\hat{\phi}_2^\e (z )  - v_* | + | \hat{\phi}_3^\e (z) - w_*| \leq \delta (\epsilon),\; \forall\, z \geq z_\e,
\eeaa
where $z_\e \geq 0$ and $\delta (\e) \to 0$ as $\e \to 0$. Finally, due to $\beta_* >0$ and applying Lemma~\ref{ODE_lemma1_ter1}, we conclude that $\hat{\phi}_1^\e \equiv 0$, a contradiction. The lemma is proved.
\end{proof}

{Finally, by elliptic estimates, for any sequence $z_n \to \infty$ and any limit $(\phi_{1,\infty}, \phi_{2,\infty}, \phi_{3,\infty} )$ of $(\phi_1, \phi_2 , \phi_3 )( \cdot + z_n)$,
then $(u,v,w) (x,t) = (\phi_{1,\infty}, \phi_{2,\infty}, \phi_{3,\infty}) (x + st)$ is an entire solution of~\eqref{main}.
Moreover, by Lemmas~\ref{lem:liminf1_ter3}, \ref{lem:liminf1_ter2} and~\ref{lem:liminf1_bis}, it satisfies the assumptions of Lemma~\ref{LE-entire-full}. It follows that the stable tail limit is the desired state~$E_c$.}

\subsection{The case of the alien weak competitor}

Next we turn to the case when the aboriginal prey is the strong competitor. Here we only manage to establish the stable tail when it is the co-existence state. Still, the following result states that the alien weak competitor always invades the environment when \eqref{positive2} and \eqref{uur} hold, as we assume throughout this subsection.
\begin{lemma}\label{lem:liminf1_ter}
It holds that $\phi_1^- = \liminf_{z \to \infty} \phi_1 (z) >0$.
\end{lemma}
The proof is exactly the same as that of Lemma~\ref{lem:liminf1_bis}, and therefore we omit it. It relies on the fact that, in Theorem~\ref{th:cs2}, we make the assumption that $\beta_* >0$.

Let us now make the additional assumptions that \eqref{co-ex} and \eqref{ODE_lyapu} hold. Then we show that $(\phi_1,\phi_2,\phi_3)$ satisfies \eqref{Bc}.} Thanks to Lemma~\ref{LE-entire-full}, it is enough to show that $\phi_i^- >0$ for $i = 1,2,3$. We already dealt with $i=1,3$, so that it only remains to prove the following lemma.
\begin{lemma}
It holds that $\phi_2^- = \liminf_{z \to \infty} \phi_2 (z) >0$.
\end{lemma}
Again, this result has actually already been proved above. Indeed, it is the same as Lemma~\ref{lem:liminf1_ter2} thanks to the fact that $\beta^* >0$, which is itself a consequence of \eqref{co-ex} and \eqref{ODE_lyapu}. This concludes the proof of Theorem~\ref{th:cs2}.

%%%%%%%%%%%%%%%%%%%%%%%%%%%%%%%%%%
\section{Non-existence of traveling waves}\label{sec:non}
\setcounter{equation}{0}

The first statement of Theorem~\ref{th:non} immediately follows from the next result.
\begin{theorem} Assume that $\beta^* >0$. For $s<s^{*}$, there is no positive solution $(\phi_1,\phi_2,\phi_3)$ of~\eqref{TWS} {and \eqref{BCv} with}
\beaa
%\lim_{z\rightarrow-\infty} (\phi_1,\phi_2,\phi_3)(z) = E^{*}\quad\text{and}\quad
\liminf_{z\rightarrow\infty} \phi_2 (z) >0 .
\eeaa
\end{theorem}
\begin{proof}
First, suppose that for some $s \leq 0$, there exists a positive solution $(\phi_1,\phi_2,\phi_3)$ {of \eqref{TWS}} satisfying the boundary condition \eqref{BCv}.
We choose $N >  1$ large enough so that if $y<-N$ then
\beaa
 1-h\phi_1(y)-\phi_2(y)-b_2\phi_3(y)>\frac{\b^{*}}{2}.
\eeaa
Now we integrate the second equation in \eqref{TWS} in $y$ from $-\infty$ to $z\leq-N$ and in $z$ from  $-\infty$ to $-N$.
Then we obtain the following contradiction
\beaa
0<\frac{r_2\b^{*}}{2}\int^{-N}_{-\infty}\int^{z}_{-\infty}\phi_2(y)dydz< -\int^{-N}_{-\infty}d_2\phi_2'(z)dz=-d_2\phi_2(-N)<0.
\eeaa

Now suppose that there exists such a traveling wave solution for {some $s\in(0,s^*)$.}
Then we pick $\e$ small enough so that $0<s<2\sqrt{d_2r_2[\b^{*}-(h+b_2)\e]}$. By the positivity and continuity of $(\phi_1,\phi_2,\phi_3)$, and the fact that
\beaa
\lim_{z\rightarrow-\infty} (\phi_1,\phi_2,\phi_3)(z) = E^{*}\quad\text{and}\quad\liminf_{z\rightarrow\infty} \phi_2 )(z)>0,
\eeaa
there are nonnegative constants $c_1$ and $c_2$ such that
\bea
&& \phi_1(z)-c_1\phi_2(z)<u^{*}+\e,\;\;\forall z \in \mathbb{R},\label{phi12 ineq}\\
&& \phi_3(z)-c_2\phi_2(z)<w^{*}+\e,\;\;\forall z \in \mathbb{R}.\label{phi32 ineq}
\eea
Using the notation $(u,v,w)(x,t) = (\phi_1,\phi_2,\phi_3)(x+st)$ and plugging \eqref{phi12 ineq} and \eqref{phi32 ineq} into the second equation of \eqref{main}, we get
\beaa
v_{t}\geq d_2v_{xx}+r_2v[\beta^{*}-(h+b_2)\e-(1+hc_1+b_2c_2)v],\;\; x\in\mathbb{R},\;t>0 .
\eeaa
The spreading theory of \cite{AW75} gives
\beaa
\liminf_{t\rightarrow\infty} v(ct,t)\geq \frac{\beta^{*}-(h+b_2)\e}{1+hc_1+b_2c_2}>0,
\eeaa
for any $|c| < 2 \sqrt{d_2 r_2 [\beta^* - (h+ b_2) \e]}$. This in particular holds true with
$$c := - \frac{s + 2 \sqrt{d_2 r_2 [\beta^* - (h+ b_2)\e] }}{2} .$$
On the other hand  $ct +st = (s-2\sqrt{d_2r_2[\beta-(h+b_2)\e]})t/2\to-\infty$ as $t\to\infty$. This implies that $v(ct,t) = \phi_2(ct+st) \to 0$ as $t\to\infty$, a contradiction.
\end{proof}
When $\beta_* >0$, one can check by the same argument that there is no positive traveling wave solution going to $E_*$ at the unstable tail limit and such that the infimum limit at $\infty$ of the first component is positive. This concludes the proof of Theorem~\ref{th:non}.

%%%%%%%%%%%%%%%%%%%%%%
%%%%%%%%%%%%%%%%%%%%%%
%%%%%%%%%%%%%%%%%%%%%%%

\section{Verification of upper-lower-solutions}\label{verul}
\setcounter{equation}{0}

%%%%%%%%%%%%%%%%%%%%%%%%%%%%%%%%%%%
\subsection{Proof of Lemma~\ref{u0w}}

It is easy to check that \eqref{order} and \eqref{lrd} hold, as well as the unstable tail limit. Therefore we focus only on the differential inequalities.

(1) $\mathcal{U}_1(z)\leq 0$ for $z\neq 0$. Recall that for $z>0$,
\beaa
\overline{\phi}_1(z) = 1,\; \underline{\phi}_2(z) = 0,\; \underline{\phi}_3(z) = 0.
\eeaa
It immediately follows that
\beaa
\mathcal{U}_1(z) =  0\quad \mbox{ for $z>0$.}
\eeaa
On the other hand, for $z<0$, we have
\beaa
\overline{\phi}_1(z) = u^{*}+b_1w^{*}e^{\l_1 z},\; \underline{\phi}_2(z)\geq 0,\; \underline{\phi}_3(z) = w^{*}(1 -e^{\l_1 z}).
\eeaa
Then, using $u^{*}+b_1w^{*}=1$, we obtain
\beaa
\mathcal{U}_1(z) &\leq & b_1w^{*}(d_1\l_1^2-s\l_1)e^{\l_1 z} +r_1(u^{*}+b_1w^{*}e^{\l_1 z})( -b_1w^{*}e^{\l_1 z}+b_1w^{*}e^{\l_1 z} )\\
                 &  =  & b_1w^{*}(d_1\l_1^2-s\l_1)e^{\l_1 z} \leq 0,
\eeaa
for $z < 0$, where the last inequality holds thanks to~\eqref{vineq}.

(2) $\mathcal{U}_2(z)\leq 0$ for $z\neq 0$. For $z>0$, we have
\beaa
\underline{\phi}_1(z)\geq 0,\; \overline{\phi}_2(z) = 1,\; \underline{\phi}_3(z) = 0.
\eeaa
Therefore $\mathcal{U}_2(z) \leq 0$ for $z>0$. For $z<0$,
\beaa
\underline{\phi}_1(z)=u^{*}(1-p_1e^{\l_1 z}),\; \overline{\phi}_2(z) = e^{\l_1 z},\; \underline{\phi}_3(z) = w^{*}(1 -e^{\l_1 z}),
\eeaa
and then
\beaa
\mathcal{U}_2(z) & =  & (d_2\l_1^2-s\l_1)e^{\l_1 z} +r_2e^{\l_1 z}[1-hu^{*}+hu^{*}p_1e^{\l_1 z}-e^{\l_1 z}-b_2w^{*}+b_2w^{*}e^{\l_1 z}]\\
                 & =  & G(\l_1 )e^{\l_1 z} + r_2e^{2\l_1 z}(hu^{*}p_1+b_2w^{*}-1)\\
                 &\leq& r_2e^{2\l_1 z}(hu^{*}+b_2w^{*}-1)\leq 0.
\eeaa
Here we used  $G(\l_1)=0$, $p_1\le 1$ from \eqref{vp} and $\b^*>0$.

(3) $\mathcal{U}_3(z)\leq 0$ for $z\neq 0$. For $z>0$, we have
\beaa
\overline{\phi}_1(z) = 1,\; \overline{\phi}_2(z) = 1,\; \overline{\phi}_3(z) = 2a-1,
\eeaa
hence
\beaa
 \mathcal{U}_3(z) &=& r_3(2a-1)[-1+a+a-(2a-1)]= 0.
\eeaa
Furthermore, for $z<0$,
\beaa
\overline{\phi}_1(z) = u^{*}+b_1w^{*}e^{\l_1 z},\; \overline{\phi}_2(z) = e^{\l_1 z},\; \overline{\phi}_3(z) = w^{*}+Ae^{\l_1 z}.
\eeaa
Using $-1+au^{*}-w^{*}=0$ and \eqref{vB}, we get
\beaa
\mathcal{U}_3(z)&=& A(d_3\l_1^2-s\l_1)e^{\l_1 z} + r_3(w^{*}+Ae^{\l_1 z})e^{\l_1 z}(ab_1w^{*}+a-A)\\
                &=& A(d_3\l_1^2-s\l_1)e^{\l_1 z} + r_3(w^{*}+Ae^{\l_1 z})e^{\l_1 z}[(1+ab_1)w^{*}- a + 1]\\
                &=& A(d_3\l_1^2-s\l_1)e^{\l_1 z}\leq 0,
\eeaa
for $z <0$, thanks to $(1+ab_1)w^{*}=a-1$ and \eqref{vineq}.

(4) $\mathcal{L}_1(z)\geq 0$ for $z\not \in \{  0,z_1 \}$. For $z>z_1\ge 0$, we get
\beaa
\underline{\phi}_1(z)= 0,\; \overline{\phi}_2(z)= 1,\; \overline{\phi}_3(z)=2a-1,
\eeaa
and thus $\mathcal{L}_1(z) =0$.

For $0<z<z_1$, we have
\beaa
\underline{\phi}_1(z)=u^{*}(1-p_1e^{\l_1 z}),\; \overline{\phi}_2(z)=1,\; \overline{\phi}_3(z) = 2a-1,
\eeaa
and, by \eqref{vineq} and the  choice of $p_1$ in \eqref{vp}, it follows that
\beaa
\mathcal{L}_1(z) &  =  & -u^{*}p_1(d_1\l_1^2-s\l_1)e^{\l_1 z} + r_1u^{*}(1-p_1e^{\l_1 z})[1-u^{*}(1-p_1e^{\l_1 z})-k-b_1(2a-1)]\\
                 &\geq & -u^{*}p_1(d_1\l_1^2-s\l_1)e^{\l_1 z} - r_1u^{*}[k+b_1(2a-1)]\\
                 &\geq &  u^{*}\left\{ -p_1(d_1\l_1^2-s\l_1) - r_1[k+b_1(2a-1)] \right\} \geq 0 .
\eeaa
Lastly, for $z<0$, recall that
\beaa
\underline{\phi}_1(z)=u^{*}(1-p_1e^{\l_1 z}),\; \overline{\phi}_2(z)=e^{\l_1 z},\; \overline{\phi}_3(z) = w^{*}+Ae^{\l_1 z}.
\eeaa
Then, using $u^{*}+b_1w^{*}=1$ and again the choice of $p_1$ in \eqref{vp}, we obtain
\beaa
\mathcal{L}_1(z) &  =  & -u^{*}p_1(d_1\l_1^2-s\l_1)e^{\l_1 z} + r_1u^{*}(1-p_1e^{\l_1 z})e^{\l_1 z}[u^{*}p_1-k-b_1A]\\
                 &\geq & -u^{*}p_1(d_1\l_1^2-s\l_1)e^{\l_1 z} + r_1u^{*}(1-p_1e^{\l_1 z})e^{\l_1z}[-k-b_1(2a-1)]\\
                 &\geq & u^{*}e^{\l_1 z}\left\{ -p_1(d_1\l_1^2-s\l_1) - r_1[k+b_1(2a-1)] \right\} \geq 0,
\eeaa
for $z<0$.

(5) $\mathcal{L}_2(z)\geq 0$ for $z\neq z_2$. For $z>z_2$, we have $\underline{\phi}_2 (z) = 0$ and thus $\mathcal{L}_2(z) = 0$.

Then, for $z<z_2<0$,
\beaa
\overline{\phi}_1(z) = u^{*}+b_1w^{*}e^{\l_1 z},\; \underline{\phi}_2(z)=e^{\l_1 z}-qe^{\mu\l_1 z},\; \overline{\phi}_3(z)=w^{*}+Ae^{\l_1 z}.
\eeaa
Using $G(\l_1)=0$ and $\b^*>0$, we get
\beaa
\mathcal{L}_2(z) &\geq & -qG(\mu\l_1)e^{\mu\l_1 z} +r_2(e^{\l_1 z}-qe^{\mu\l_1 z})(-hb_1w^{*}e^{\l_1z}-e^{\l_1z}-b_2Ae^{\l_1z})\\
                 &\geq & -qG(\mu\l_1)e^{\mu\l_1 z} -r_2e^{2\l_1 z}(hb_1w^{*}+1+b_2A)\\
                 & \geq  & e^{\mu\l_1 z}[ -qG(\mu\l_1)-r_2e^{(2-\mu)\l_1 z}(hb_1w^{*}+1+b_2A) ]\\
                 &\geq & e^{\mu\l_1 z}[ -qG(\mu\l_1)-r_2(hb_1w^{*}+1+b_2A) ]\geq 0,
\eeaa
for $z<z_2$, by the choice of $\mu$ in \eqref{vmu}, which ensures that $2 - \mu >0$ and $G(\mu\l_1)< 0$, and the choice of $q$ in \eqref{vq}.

(6) $\mathcal{L}_3(z)\geq 0$ for $z\neq 0$. For $z>0$, $\underline{\phi}_3(z)= 0$ gives $\mathcal{L}_3(z) =0$.

For $z <0$, we have
\beaa
\underline{\phi}_1(z)=u^{*}(1-p_1e^{\l_1 z}),\; \underline{\phi}_2(z)  \geq 0,\; \underline{\phi}_3(z) = w^{*}(1 -e^{\l_1 z})\le w^*.
\eeaa
Then, using $p_1\le 1$, $-1+au^{*}-w^{*}=0$ and \eqref{vd}, we get
\beaa
\mathcal{L}_3(z) & \geq  & -w^{*}(d_3\l_1^2-s\l_1)e^{\l_1 z}+r_3\underline{\phi}_3(z)e^{\l_1 z}(-au^{*}p_1 +w^{*})\\
                 &\geq& -w^{*}(d_2\l_1^2-s\l_1)e^{\l_1 z}-r_3\underline{\phi}_3(z)e^{\l_1 z}\\
                 &\geq& w^*e^{\l_1 z}(r_2\beta^*-r_3)\ge 0,
\eeaa
for $z<0$.
%
%Lastly, for $z<z_2<0$, we have
%\beaa
%\underline{\phi}_1(z)=u^{*}(1-p_1e^{\l_1 z}),\; \underline{\phi}_2(z) = e^{\l_1 z}-qe^{\mu\l_1 z},\; \underline{\phi}_3(z) = w^{*}(1 -e^{\l_1 z}).
%\eeaa
%Using $-1+au^{*}-w^{*}=0$ again and recalling the definition of $z_2$, which is such that
%\beaa
%qe^{(\mu-1)\l_1 z_2} = 1 -\delta_2e^{-\l_1 z_2},
%\eeaa
%{\blue we obtain
%\beaa
%\mathcal{L}_3(z) &=& -w^{*}(d_3\l_1^2-s\l_1)e^{\l_1 z}+r_3\underline{\phi}_3(z)e^{\l_1 z}[-au^{*}p_1+a-aqe^{(\mu-1)\l_1 z}+w^{*}]\\
%&\ge&-w^{*}(d_3\l_1^2-s\l_1)e^{\l_1 z}+r_3\underline{\phi}_3(z)e^{\l_1 z}(-au^{*}p_1+a\delta_2e^{-\l_1 z_2}+w^{*}),
%\eeaa
%for $z < z_2$. Then the same argument as in the case $z\in(z_2,0)$ above gives $\mathcal{L}_3(z)\geq 0$ for~$z<z_2$.}
This completes the proof of Lemma~\ref{u0w}.

%%%%%%%%%%%%%%%%%%%%%%%%%

%%%%%%%%%%%%%%%
%%%%%%%%%%%%%%%%%%
\subsection{Proof of Lemma~\ref{u00w}}

Here $s= s^*$. As before, we only deal with the differential inequalities.

(1) $\mathcal{U}_1(z)\leq 0$ for $z\neq -2/\l_1$. For $z>-2/\l_1$, then $\overline{\phi}_1(z) = 1$, $\underline{\phi}_2(z)\geq 0$ and $\underline{\phi}_3(z)=0$, so that $\mathcal{U}_1(z)\leq 0$.

For $z<-2/\l_1$, we have that
\beaa
\overline{\phi}_1(z)=u^{*}+L^{*}b_1w^{*}(-z)e^{\l_1z},\;\underline{\phi}_2(z)\geq 0,\;\underline{\phi}_3(z)=w^{*}[1-L^{*}(-z)e^{\l_1z}].
\eeaa
Then
\beaa
\mathcal{U}_1(z) &\leq & L^{*}b_1w^*(-2d_1\l_1+s)e^{\l_1z}+L^{*}b_1w^*(d_1\l_1^2-s\l_1 )(-z)e^{\l_1z}\\
                 &  =  & -r_2 L^{*}b_1w^*\b^{*}(-z)e^{\l_1z}\le 0,
\eeaa
for $z<-2/\l_1$, by using the first part of \eqref{vvvd0}, \eqref{l} with $s= s^*$ and $\b^{*}>0$.

(2) $\mathcal{U}_2(z)\leq 0$ for $z\neq -2/\l_1$. For $z>-2/\l_1$, we have that $\underline{\phi}_1(z)\geq 0$, $\overline{\phi}_2(z) = 1$ and $\underline{\phi}_3(z) = 0$, hence $\mathcal{U}_2(z)\leq 0$.

For $z<-2/\l_1$, due to $z_1 > - 2 /\l_1$, we have
\beaa
\underline{\phi}_1(z)=u^{*}[1-p_1L^{*}(-z)e^{\l_1z}],\;\overline{\phi}_2(z) = L^{*}(-z)e^{\l_1z},\;\underline{\phi}_3(z) =w^{*}[1-L^{*}(-z)e^{\l_1z}].
\eeaa
Then
\beaa
\mathcal{U}_2(z) &  = & L^*(-2d_2\l_1+s)e^{\l_1 z}+L^{*}(d_2\l_1^2-s\l_1)(-z)e^{\l_1z}\\
                 && +r_2L^{*}(-z)e^{\l_1z}[\b^{*}+hu^{*}p_1L^{*}(-z)e^{\l_1z}-L^{*}(-z)e^{\l_1z}+b_2w^{*}L^{*}(-z)e^{\l_1z}]\\
                 &  = & r_2[L^{*}(-z)e^{\l_1z}]^2(-1+hu^{*}p_1+b_2w^{*})\\
                 &\leq& -r_2[L^{*}(-z)e^{\l_1z}]^2(1-hu^{*}-b_2w^{*})= -r_2[L^{*}(-z)e^{\l_1z}]^2\b^{*}\le 0,
\eeaa
for $z<-2/\l_1$, by using $p_1\leq 1$ and $\b^{*}>0$.

(3) $\mathcal{U}_3(z)\leq 0$ for $z\neq -2/\l_1$. For $z>-2/\l_1$, then $\overline{\phi}_1(z)=1$, $\overline{\phi}_2(z)=1$ and $\overline{\phi}_3(z)=2a-1$ and hence $\mathcal{U}_3(z) = 0$.

For $z<-2/\l_1$,
\beaa
\overline{\phi}_1(z)=u^{*}+L^{*}b_1w^{*}(-z)e^{\l_1z},\;\overline{\phi}_2(z)=L^{*}(-z)e^{\l_1z},\;\overline{\phi}_3(z)=w^{*}+L^{*}A(-z)e^{\l_1z}.
\eeaa
Then
\beaa
\mathcal{U}_3(z) & = & L^{*} A (-2d_3\l_1+s)e^{\l_1z} +L^{*} A (d_3\l_1^2-s\l_1)(-z)e^{\l_1z}\\
                 &\quad&+r_3\overline{\phi}_3(z)[aL^{*} b_1w^{*}(-z)e^{\l_1z}+aL^{*}(-z)e^{\l_1z}-L^{*}A(-z)e^{\l_1z}]\\
                 & = & L^{*} A(-2d_3\l_1+s)e^{\l_1z} +L^{*} A(d_3\l_1^2-s\l_1)(-z)e^{\l_1z}\\
                 &\quad&+r_3\overline{\phi}_3(z)[(1+ab_1)w^{*}- ( a-1)]L^{*}(-z)e^{\l_1z}\\
                 & = & L^{*} A (-2d_3\l_1+s)e^{\l_1z} +L^{*} A (d_3\l_1^2-s\l_1)(-z)e^{\l_1z}\leq 0,
\eeaa
for $z<-2/\l_1$, using $A = 2 a - 1 - w^* >0$ and since $-2d_3\l_1+s\leq 0$ and $d_3\l_1^2-s\l_1\leq 0$, by the first part of~\eqref{vvvd0}.

(4) $\mathcal{L}_1(z)\geq 0$ for $z\not \in \{ -2/\l_1,z_1 \}$. For $z>z_1$, $\mathcal{L}_1(z)=0$ by $\underline{\phi}_1=0$. Next, for $-2/\l_1<z<z_1$, we have
\beaa
\underline{\phi}_1(z) = u^{*}[1-p_1L^{*}(-z)e^{\l_1z}],\;\overline{\phi}_2(z)=1,\;\overline{\phi}_3(z)=2a-1.
\eeaa
Then
\beaa
\mathcal{L}_1(z)  & =   & -u^{*}p_1L^{*}(-2d_1\l_1+s)e^{\l_1z} -u^{*}p_1L^{*}(d_1\l_1^2-s\l)(-z)e^{\l_1z}\\
                  &\quad& +r_1\underline{\phi}_1(z)[1-\underline{\phi}_1(z)-k-b_1(2a-1)]\\
                  &\geq & -u^{*}p_1(d_1\l_1^2-s\l_1)-r_1u^{*}[k+b_1(2a-1)]\\
                  & =   & u^{*}\{-p_1(d_1\l_1^2-s\l_1)-r_1[k+b_1(2a-1)]\}\geq 0,
\eeaa
by the first part of~\eqref{vvvd0}, $L^*(-z)e^{\l_1z}\ge 1$ for $z\in(-2/\l_1,z_1)$, $d_1 \lambda_1^2 - s \lambda < 0$ and the choice of $p_1$ in \eqref{vp_bis}.

Lastly, for $z<-2/\l_1$,
\beaa
\underline{\phi}_1(z) = u^{*}[1-p_1L^{*}(-z)e^{\l_1z}],\;\overline{\phi}_2(z)=L^{*}(-z)e^{\l_1z},\;\overline{\phi}_3(z)=w^{*}+L^{*}A(-z)e^{\l_1z},
\eeaa
and thus
\beaa
\mathcal{L}_1(z)  & =   & -u^{*}p_1L^{*}(-2d_1\l_1+s)e^{\l_1z} -u^{*}p_1L^{*}(d_1\l_1^2-s\l_1)(-z)e^{\l_1z}\\
                  &\quad& +r_1u^{*}[1-p_1L^{*}(-z)e^{\l_1z}][ u^{*}p_1L^{*}(-z)e^{\l_1z}-kL^{*}(-z)e^{\l_1z}-b_1 L^{*} A(-z)e^{\l_1z} ]\\
 & = & { -u^{*}p_1L^{*}(d_1\l_1^2-s\l_1)(-z)e^{\l_1z}  } \\
 &&  + r_1 u^* [ 1 - p_1 L^* (-z) e^{\l_1 z} ] [u^* p_1 + b_1 w^* - k -b_1 (2a -1)] L^* (-z) e^{\l_1 z} \\
                  &\geq & -u^{*}p_1L^{*}(d_1\l_1^2-s\l_1)(-z)e^{\l_1z} + r_1u^{*}[-k-b_1(2a-1)]L^{*}(-z)e^{\l_1z}\\
                  & =   & u^{*}\{-p_1(d_1\l_1^2-s\l_1)-r_1[k+b_1(2a-1)]\}L^{*}(-z)e^{\l_1z}\geq 0,
\eeaa
where we again used the first part of~\eqref{vvvd0} and \eqref{vp_bis}, {and in particular the fact that $u^* p_1 + b_1 w^* - k \leq 1 - k <0$.}

(5) $\mathcal{L}_2(z)\geq 0$ for $z\neq z_2$. For $z>z_2$, then $\overline{\phi}_2 (z) = 0$ and $\mathcal{L}_2(z) =  0$.

For $z<z_2$, and since $z_2 < -2/\l_1$, we have $\overline{\phi}_1(z)=u^{*}+L^{*}b_1 w^{*}(-z)e^{\l_1z}$, as well as
\beaa
\underline{\phi}_2(z)=[L^{*}(-z)-q(-z)^{1/2}]e^{\l_1z},\;\overline{\phi}_3(z)=w^{*}+L^{*}A(-z)e^{\l_1z}.
\eeaa
It follows that, for $z<z_2$,
\beaa
\mathcal{L}_2(z) &  = & q\frac{d_2}{4}(-z)^{-3/2}e^{\l_1z}+\underline{\phi}_2(z)(d_2\l_1^2-s\l_1)\\
                 && +r_2\underline{\phi}_2(z)[\b^{*}-hL^{*}b_1w^{*}(-z)e^{\l_1z}-L^{*}(-z)e^{\l_1z}+q(-z)^{1/2}e^{\l_1z}-b_2L^{*}A(-z)e^{\l_1z}]\\
                 &\geq& q\frac{d_2}{4}(-z)^{-3/2}e^{\l_1z}+r_2[L^{*}(-z)e^{\l_1z}]^2(-hb_1w^{*}-1-b_2A)\\
                 & =  & \frac{d_2}{4}(-z)^{-3/2}e^{\l_1z} \left[ q-\frac{4}{d_2}r_2(L^{*})^2(-z)^{7/2}e^{\l_1z}(hb_1w^{*}+1+b_2A) \right]\\
                 &\geq& \frac{d_2}{4}(-z)^{-3/2}e^{\l_1z} \left[q-\frac{4}{d_2}r_2(L^{*})^2M(hb_1w^{*}+1+b_2A) \right]\geq 0,
\eeaa
by the choice of $q$ in \eqref{vvq}, where we have used
\beaa
(-z)^{7/2}e^{\l_1z}\leq M:=\left(\frac{7}{2\l_1e}\right)^{7/2},\;\forall\, z<0.
\eeaa

(6) $\mathcal{L}_3(z)\geq 0$ for $z\neq - 2/ \lambda_1$. For $z>-2/\l_1$, due to $\underline{\phi}_3(z)=0$, we immediately get that $\mathcal{L}_3(z)=0$.

For $z<-2/\l_1$, we also have $z < z_1$ and
\beaa
\underline{\phi}_1(z)=u^{*}[1-p_1L^{*}(-z)e^{\l_1z}],\;\underline{\phi}_2(z) \geq 0,\;\underline{\phi}_3(z)=w^{*}[1-L^{*}(-z)e^{\l_1z}].
\eeaa
Then
\beaa
\mathcal{L}_3(z)  &  \geq  & -w^{*}L^{*}(-2d_3\l_1+s)e^{\l_1z}-w^{*}L^{*}(d_3\l_1^2-s\l_1)(-z)e^{\l_1z}\\
                  &\quad&+r_3\underline{\phi}_3(z)[-au^{*}p_1L^{*}(-z)e^{\l_1z}+w^{*}L^{*}(-z)e^{\l_1z}]\\
& \geq & [ - (d_3\l_1^2-s\l_1) - r_3  a u^* p_1 +  r_3  w^* ]w^*  L^* (-z)e^{\l_1z}\\
& \geq &[ - (d_3\l_1^2-s\l_1) - r_3   ]w^*  L^* (-z)e^{\l_1z},
\eeaa
using $s = 2 d_2 \lambda_1 \leq 2 d_3 \lambda_1$, $p_1 \leq 1$ and $au^* - w^* = 1$. Now notice that
$$d_3 \lambda_1^2 - s \lambda_1 = \left( \frac{d_3}{d_2} - 2 \right) r_2 \beta^* \leq - r_3, $$
due to the second part of~\eqref{vvvd0}. It follows that $\mathcal{L}_3 (z) \geq 0$ for $z < - 2 / \l_1$.
This completes the proof of this lemma.

%%%%%%%%%%%%%%%%%%

\subsection{Proof of Lemma~\ref{0vw}} We now turn to the case when the invaded state is $E_*$, first when $s > s_*$.

(1) $\mathcal{U}_1(z)\leq 0$ for $z\neq 0$.
For $z>0$, then $\overline{\phi}_1(z) = 1,\;\underline{\phi}_2(z)\geq 0,\;\underline{\phi}_3(z)=0$ and it follows that $\mathcal{U}_1(z) \leq 0$.

For $z<0$,
$$\overline{\phi}_1(z) = e^{\sigma_1z},\;\underline{\phi}_2(z)=v_{*}(1-p_2e^{\sigma_1z}),\;\underline{\phi}_3(z)=w_{*}(1-e^{\sigma_1z}).$$
In that case,
\beaa
\mathcal{U}_1(z) &  = & (d_1\sigma_1^2-s\sigma_1)e^{\sigma_1z} +r_1e^{\sigma_1z}[ 1-e^{\sigma_1z}-kv_{*}(1-p_2e^{\sigma_1z})-b_1w_{*}(1-e^{\sigma_1z}) ]\\
                 &  = & H(\sigma_1)e^{\sigma_1z} +r_1e^{2\sigma_1z}(-1+kv_{*}p_2+b_1w_{*})\\
                 &\leq& r_1e^{2\sigma_1z}(-1+kv_{*}+b_1w_{*}) = -r_1e^{2\sigma_1z}\b_{*}\leq 0 ,
\eeaa
by $p_2\le 1$ and $\b_{*}>0$.

(2) $\mathcal{U}_2(z)\leq 0$ for $z\neq 0$.
For $z>0$, $\underline{\phi}_1(z) = 0$, $\overline{\phi}_2(z) = 1$, $\underline{\phi}_3(z) = 0$ and so $\mathcal{U}_2(z) = 0$.

For $z<0$,
$$\underline{\phi}_1(z)\geq 0,\;\overline{\phi}_2(z) = v_{*}+b_2w_{*}e^{\sigma_1z},\;\underline{\phi}_3(z) =w_{*}(1-e^{\sigma_1z}).$$
Using $1-v_{*}-b_2w_{*}=0$ and \eqref{uineq}, we get
\beaa
\mathcal{U}_2(z)  \le b_2w_{*} (d_2\sigma_1^2-s\sigma_1)e^{\sigma_1z}  \le 0\;\mbox{ for $z<0$.}
\eeaa

(3) $\mathcal{U}_3(z)\leq 0$ for $z\neq 0$.
For $z>0$, $\overline{\phi}_1(z) =1,\;\overline{\phi}_2(z) =1,\;\overline{\phi}_3(z) =2a-1$ and hence $\mathcal{U}_3(z)= 0$.

For $z<0$, we have
\beaa
\overline{\phi}_1(z) =e^{\sigma_1z},\;\overline{\phi}_2(z) =v_{*}+b_2w_{*}e^{\sigma_1z},\;\overline{\phi}_3(z) =w_{*}+Be^{\sigma_1z}.
\eeaa
Using $-1+av_{*}-w_{*} = 0$ and again \eqref{uineq}, we obtain
\beaa
\mathcal{U}_3(z) & = & B(d_3\sigma_1^2-s\sigma_1)e^{\sigma_1z} + r_3\overline{\phi}_3(z)[ae^{\sigma_1z}+ab_2w_{*}e^{\sigma_1z}-Be^{\sigma_1z}] \\
& \leq &  r_3\overline{\phi}_3(z)[ae^{\sigma_1z}+ab_2w_{*}e^{\sigma_1z}-Be^{\sigma_1z}].
\eeaa
Now note that
\beaa
ae^{\sigma_1z}+ab_2w_{*}e^{\sigma_1z}-Be^{\sigma_1z} = e^{\sigma_1z}[a+(1+ab_2)w_{*}-(2a-1)] = 0,
\eeaa
since $B = (2a -1) - w_*$ and $(1+ab_2)w_{*}=a-1$.
Hence we deduce that $\mathcal{U}_3(z) \leq 0$ for $z<0$.

(4) $\mathcal{L}_1(z)\geq 0$ for $z\neq z_0$.
For $z>z_0$, we have $\underline{\phi}_1(z) = 0,\;\overline{\phi}_2(z)\leq 1,\;\overline{\phi}_3(z)\leq 2a-1$, therefore $\mathcal{L}_1(z)=0$.

For $z<z_0<0$,
$$\underline{\phi}_1(z) = e^{\sigma_1z}-qe^{\mu\sigma_1z},\;\overline{\phi}_2(z)=v_{*}+b_2w_{*}e^{\sigma_1z},\;\overline{\phi}_3(z)=w_{*}+Be^{\sigma_1z}.$$
Then
\beaa
\mathcal{L}_1(z) &  = & -qH(\mu\sigma_1)e^{\mu\sigma_1z}\\
                  &\quad& +r_1(e^{\sigma_1z}-qe^{\mu\sigma_1z})[-e^{\sigma_1z}+qe^{\mu\sigma_1z}-kb_2w_{*}e^{\sigma_1z}-b_1Be^{\sigma_1z}]\\
                 &\geq& -qH(\mu\sigma_1)e^{\mu\sigma_1z}+r_1e^{\sigma_1z}(-e^{\sigma_1z}-kb_2w_{*}e^{\sigma_1z}-b_1Be^{\sigma_1z})\\
%                 &  = & -qH(\mu\sigma_1)e^{\mu\sigma_1z}-r_1e^{2\sigma_1z}(1+kb_2w_{*}+b_1B)\\
                 &  = & e^{\mu\sigma_1z}[-qH(\mu\sigma_1)-r_1e^{(2-\mu)\sigma_1z}(1+kb_2w_{*}+b_1B)]\\
                 &\geq& e^{\mu\sigma_1z}[-qH(\mu\sigma_1)-r_1(1+kb_2w_{*}+b_1B)]\geq 0,
\eeaa
by the choice of $\mu$ in \eqref{uumu} and $q$ in \eqref{uuq}.

(5) $\mathcal{L}_2(z)\geq 0$ for $z\not \in \{ 0, z_2\}$. First, for $z>z_2 \geq 0$, we have $\mathcal{L}_2(z)= 0$ since $\underline{\phi}_2(z) = 0$.

Next, for any $0<z<z_2$, we have $\overline{\phi}_1(z) = 1,\;\underline{\phi}_2(z) = v_{*}(1-p_2e^{\sigma_1z})<1,\;\overline{\phi}_3(z) = 2a-1 ,$
and then
\beaa
\mathcal{L}_2(z) &  = & -v_{*}p_2(d_2\sigma_1^2-s\sigma_1)e^{\sigma_1z} + r_2v_{*}(1-p_2e^{\sigma_1z})[1-h-v_{*}(1-p_2e^{\sigma_1z})-b_2(2a-1)]\\
                 &\geq& -v_{*}p_2(d_2\sigma_1^2-s\sigma_1)e^{\sigma_1z} + r_2v_{*}(1-p_2e^{\sigma_1z})[-h-b_2(2a-1)]\\
                 &\geq& -v_{*}p_2(d_2\sigma_1^2-s\sigma_1) - r_2v_{*}[h+b_2(2a-1)]\\
                 &  = &  v_{*}\left\{ -p_2(d_2\sigma_1^2-s\sigma_1)-r_2[h+b_2(2a-1)] \right\}\geq 0 ,
\eeaa
by \eqref{uineq} and the choice of $p_2$ in \eqref{uup}.

Lastly, for $z<0$, then $\overline{\phi}_1(z) = e^{\sigma_1z},\;\underline{\phi}_2(z) = v_{*}(1-p_2e^{\sigma_1z}),\;\overline{\phi}_3(z) = w_{*}+Be^{\sigma_1z}$. It follows that
\beaa
\mathcal{L}_2(z)  & =  & -v_{*}p_2(d_2\sigma_1^2-s\sigma_1)e^{\sigma_1z} + r_2v_{*}(1-p_2e^{\sigma_1z})[-he^{\sigma_1z}+v_{*}p_2e^{\sigma_1z}-b_2Be^{\sigma_1z}]\\
                  &\geq& -v_{*}p_2(d_2\sigma_1^2-s\sigma_1)e^{\sigma_1z} + r_2v_{*}(1-p_2e^{\sigma_1z})[-he^{\sigma_1z}-b_2(2a-1)e^{\sigma_1z}]\\
                  &\geq& -v_{*}p_2(d_2\sigma_1^2-s\sigma_1)e^{\sigma_1z} - r_2v_{*}[h+b_2(2a-1)]e^{\sigma_1z}\\
                  & =  & v_{*}e^{\sigma_1z}\left\{ -p_2(d_2\sigma_1^2-s\sigma_1) - r_2[h+b_2(2a-1)]\right\}\geq 0,
\eeaa
where we used $1 - v_* - b_2 w_* = 0$, and again \eqref{uineq} and the choice of $p_2$ in \eqref{uup}.

(6) $\mathcal{L}_3(z)\geq 0$ for $z\neq 0$.
For $z>0$, since $\underline{\phi}_1(z) = 0$, we immediately get that $\mathcal{L}_3(z)= 0$.

For $z<0$, then
$$\underline{\phi}_1(z)\ge 0,\underline{\phi}_2(z)=v_{*}(1-p_2e^{\sigma_1z}),\;\underline{\phi}_3(z)=w_{*}(1-e^{\sigma_1z}).$$
Using $-1+av_{*}-w_{*} = 0$, $p_2\le 1$, $d_3\le d_1$ and \eqref{uur}, we obtain
\beaa
\mathcal{L}_3(z) & \ge  & -w_{*}(d_3\sigma_1^2-s\sigma_1)e^{\sigma_1z} +r_3\underline{\phi}_3(-av_{*}p_2+w_{*})e^{\sigma_1z}\\
                 & \ge  & -w_{*}(d_1\sigma_1^2-s\sigma_1)e^{\sigma_1z}-r_3\underline{\phi}_3e^{\sigma_1z}\\
                 &\geq& w_{*}(r_1\beta_*-r_3)e^{\sigma_1z}\geq 0,
\eeaa
for $z<0$. This completes the proof of Lemma~\ref{0vw}.

%%%%%%%%%%%%%%%%%

%%%%%%%%%%%%%%%
\subsection{Proof of Lemma~\ref{00vw}} Finally we consider the case when the invaded state is $E_*$ and the speed $s = s_*$.

(1) $\mathcal{U}_1(z)\leq 0$ for $z\neq -2/\sigma_1$. For $z>-2/\sigma_1$,
\beaa
\overline{\phi}_1(z) = 1,\;\underline{\phi}_2(z)\geq 0,\;\underline{\phi}_3(z)=0,
\eeaa
hence $\mathcal{U}_1(z)\leq 0$.

For $z<-2/\sigma_1$,
\beaa
\overline{\phi}_1(z) = L_{*}(-z)e^{\sigma_1z},\;\underline{\phi}_2(z)=v_{*}[1-p_2L_{*}(-z)e^{\sigma_1z}],\;\underline{\phi}_3(z)=w_{*}[1-L_{*}(-z)e^{\sigma_1z}].
\eeaa
Then
\beaa
\mathcal{U}_1(z) & =  & L_{*}(-2 d_1 \sigma_1+s) e^{\sigma_1z}+L_{*}(d_1\sigma_1^2-s\sigma_1)(-z)e^{\sigma_1z}\\
                 &\quad& +r_1L_{*}(-z)e^{\sigma_1z}[ \b_{*}-L_{*}(-z)e^{\sigma_1z} +kv_{*}p_2L_{*}(-z)e^{\sigma_1z} + b_1 L_* w_{*}(-z)e^{\sigma_1z} ]\\
                 &\leq& r_1L_{*}^2(-z)^2e^{2\sigma_1z}(-1 +kv_{*}+b_1w_{*})\\
                 & =  & -r_1L_{*}^2(-z)^2e^{2\sigma_1z}\b_{*}<0 ,
\eeaa
using $p_2\leq 1$ and $\b_{*} = 1 - kv_* - b_1 w_* >0$.

(2) $\mathcal{U}_2(z)\leq 0$ for $z\neq -2/\sigma_1$. For $z>-2/\sigma_1$,
\beaa
\underline{\phi}_1(z) \ge 0,\;\overline{\phi}_2(z) = 1,\;\underline{\phi}_3(z) = 0,
\eeaa
and so $\mathcal{U}_2(z) \le 0$.

For $z<-2/\sigma_1$, we have
\beaa
\underline{\phi}_1(z)\geq 0 ,\;\overline{\phi}_2(z) = v_{*}+L_{*}b_2w_{*}(-z)e^{\sigma_1z},\;\underline{\phi}_3(z) =w_{*}[1-L_{*}(-z)e^{\sigma_1z}].
\eeaa
Then, we get
\beaa
\mathcal{U}_2(z) &\leq& L_{*}(-2d_2\sigma_1+s)e^{\sigma_1z} + L_{*}(d_2\sigma_1^2-s\sigma_1)(-z)e^{\sigma_1z}\\
 & =  & -r_1L_{*}\b_{*}(-z)e^{\sigma_1z}\leq 0 ,
\eeaa
using $v_*+b_2w_*=1$, $d_1 = d_2$ and again $\beta_* >0$.

(3) $\mathcal{U}_3(z)\leq 0$ for $z\neq -2/\sigma_1$. For $z>-2/\sigma_1$, we have $\overline{\phi}_1(z) =1$, $\overline{\phi}_2(z) =1$, $\overline{\phi}_3(z) =2a-1$, and so $\mathcal{U}_3(z)=0$.

For $z<-2/\sigma_1$,
\beaa
\overline{\phi}_1(z) =L_{*}(-z)e^{\sigma_1z},\;\overline{\phi}_2(z) =v_{*}+L_{*}b_2w_{*}(-z)e^{\sigma_1z},\;\overline{\phi}_3(z) =w_{*}+L_{*}B(-z)e^{\sigma_1z}.
\eeaa
Then
\beaa
\mathcal{U}_3(z) &=& L_{*} B (-2d_3\sigma_1+s)e^{\sigma_1z}+L_{*} B (d_3\sigma_1^2-s\sigma_1)(-z)e^{\sigma_1z}\\
                                  &\quad& + r_3\overline{\phi}_3(z)[1-a+(1+ab_2)w_{*}]L_{*}(-z)e^{\sigma_1z}\\
                 &=& L_{*} B (-2d_3\sigma_1+s)e^{\sigma_1z}+L_{*} B (d_3\sigma_1^2-s\sigma_1)(-z)e^{\sigma_1z}\leq 0 ,
\eeaa
for $z<-2/\sigma_1$, using $B= (2a -1) -w_*$ and $(1+ab_2) w_*= a-1$, as well as $-2d_3\sigma_1+s\leq 0$ and $d_3\sigma_1^2-s\sigma_1\leq 0$ which are due to \eqref{uuud0}.

(4) $\mathcal{L}_1(z)\geq 0$ for $z\neq z_0$. For $z>z_0$, $\underline{\phi}_1(z)=0$ and so $\mathcal{L}_1(z) = 0$.

For $z<z_0 \leq -2/\sigma_1$, there holds $\underline{\phi}_1(z) = [L_{*}(-z)-q(-z)^{1/2}]e^{\sigma_1z}$ and
\beaa
\overline{\phi}_2(z)=v_{*}+L_{*}b_2w_{*}(-z)e^{\sigma_1z},\;\overline{\phi}_3(z)=w_{*}+L_{*}B(-z)e^{\sigma_1z}.
\eeaa
It follows that
\beaa
\mathcal{L}_1(z) &   = & \frac{d_1}{4}q(-z)^{-3/2}e^{\sigma_1z} + \underline{\phi}_1(z)(d_1\sigma_1^2-s\sigma_1)\\
                 &\quad&+r_1\underline{\phi}_1(z)[\b_{*}-\underline{\phi}_1(z)-kL_{*}b_2w_{*}(-z)e^{\sigma_1z}-b_1 L_{*} B(-z)e^{\sigma_1z}]\\
                 &\geq & \frac{d_1}{4}q(-z)^{-3/2}e^{\sigma_1z}\\
                 &\quad&+r_1[L_{*}(-z)-q(-z)^{1/2}]e^{\sigma_1z}[-L_{*}(-z)e^{\sigma_1 z}-kL_{*}b_2w_{*}(-z)e^{\sigma_1z}-b_1 L_{*} B (-z)e^{\sigma_1z}]\\
                 &\geq & \frac{d_1}{4}q(-z)^{-3/2}e^{\sigma_1z}+r_1L_{*}^2(-z)^2e^{2\sigma_1z}(-1-kb_2w_{*}-b_1B)\\
                 &  =  & (-z)^{-3/2}e^{\sigma_1z} \left[ q\frac{d_1}{4} -  r_1L_{*}^2(-z)^{7/2}e^{\sigma_1z}(1+kb_2w_{*}+b_1B) \right]\\
                 &\geq & (-z)^{-3/2}e^{\sigma_1z}\left[ q\frac{d_1}{4} -  r_1ML_{*}^2(1+kb_2w_{*}+b_1B) \right]\geq 0 ,
\eeaa
for $z<z_0$, by the choice of $q$ in \eqref{uuuq} and $(-z)^{7/2}e^{\sigma_1 z}\le M$ for all $z<0$.

(5) $\mathcal{L}_2(z)\geq 0$ for $z\not \in \{ -2/\sigma_1, z_2 \}$. For $z>z_2$, we have $\underline{\phi}_2(z)=0$ and so $\mathcal{L}_2(z) =0$.

Then, for $-2/\sigma_1<z<z_2$,
\beaa
\overline{\phi}_1(z) = 1,\;\underline{\phi}_2(z) = v_{*}[1-p_2L_{*}(-z)e^{\sigma_1z}],\;\overline{\phi}_3(z) = 2a-1,
\eeaa
and, using $s=2d_2\sigma_1$,
\beaa
\mathcal{L}_2(z) &  = & - v_{*}p_{2}L_{*}(d_2\sigma_1^2-s\sigma_1)(-z)e^{\sigma_1z}+r_2\underline{\phi}_2(z)[1-h-\underline{\phi}_2(z)-b_2(2a-1)]\\
                 &\geq& -v_{*}p_{2}L_{*}(d_2\sigma_1^2-s\sigma_1)(-z)e^{\sigma_1z} - r_2v_{*}[h+b_2(2a-1)].
\eeaa
Since $L_{*}(-z)e^{\sigma_1z}>1$ for all $-2/\sigma_1<z<z_2$, we obtain that
\beaa
\mathcal{L}_2(z) &\geq & v_{*}\{-p_{2}(d_2\sigma_1^2-s\sigma_1) - r_2[h+b_2(2a-1)]\}\geq 0 ,
\eeaa
for $z\in(-2/\sigma_1,z_2)$, by our choice of $p_2$.

Next, for $z<-2/\sigma_1$,
\beaa
\overline{\phi}_1(z) = L_{*}(-z)e^{\sigma_1z},\;\underline{\phi}_2(z) = v_{*}[1-p_2L_{*}(-z)e^{\sigma_1z})],\;\overline{\phi}_3(z) = w_{*}+L_{*}B(-z)e^{\sigma_1z}.
\eeaa
Then we compute
\beaa
\mathcal{L}_2(z) &   = &  - v_{*}p_{2}L_{*}(d_2\sigma_1^2-s\sigma_1)(-z)e^{\sigma_1z}\\
                  &\quad& +r_2 v_{*}[1-p_2L_{*}(-z)e^{\sigma_1z}][-hL_{*}(-z)e^{\sigma_1z}+v_{*}p_2L_{*}(-z)e^{\sigma_1z}-b_2 L_{*} B (-z)e^{\sigma_1z}]\\
                 &\geq &  - v_{*}p_{2}L_{*}(d_2\sigma_1^2-s\sigma_1)(-z)e^{\sigma_1z}-r_2 v_{*}[h+b_2(2a-1)]L_{*}(-z)e^{\sigma_1z}\\
                 &   = &  v_{*}\{-p_{2}(d_2\sigma_1^2-s\sigma_1)-r_2[h+b_2(2a-1)]\}L_{*}(-z)e^{\sigma_1z}\geq 0 ,
\eeaa
using $v_*+b_2w_*=1$, $s=2d_2\sigma_1$, and again our choice of $p_2$.

(6) $\mathcal{L}_3(z)\geq 0$ for $z\neq -2/\sigma_1$. For $z>-2/\sigma_1$, we have $\underline{\phi}_3(z)=0$ and so $\mathcal{L}_3(z)=0$.

For $z<-2/\sigma_1$,
\beaa
\underline{\phi}_1(z) \ge 0, \; \underline{\phi}_2(z)=v_*[1-p_2L_{*}(-z)e^{\sigma_1z}],\;\underline{\phi}_3(z)=w_{*}[1-L_{*}(-z)e^{\sigma_1z}] .
\eeaa
Then, using $av_*-w_*=1$, $p_2\le 1$, and $s= 2 d_1 \sigma_1 \leq d_3 \sigma_1$ by \eqref{uuud0}, we get
\beaa
\mathcal{L}_3(z) & \ge  &  - w_* L_*  ( -2 d_3 \sigma_1 + s)e^{\sigma_1 z}  -w_{*}L_{*}(d_3\sigma_1^2-s\sigma_1)(-z)e^{\sigma_1z} \\
& & +r_3\underline{\phi}_3(z)(-av_{*}  +w_*)L_{*}(-z)e^{\sigma_1z}\\
& \ge  & [ - (d_3 \sigma_1^2  - s \sigma_1) - r_3]  w_* L_* (-z) e^{\sigma_1 z}.
\eeaa
Due to the second part of \eqref{uuud0}, one may infer that $\mathcal{L}_3 (z) \geq 0$ for $z<-2/\sigma_1$.
The proof of this lemma is thus completed.

\bigskip
%%%%%%%%%%%%%%%%%%

\noindent{\bf Acknowledgements}

\noindent{The first author (YSC) and the third author (JSG) were partially supported by the Ministry of Science and Technology of Taiwan under the grants 108-2811-M-032-504 and 108-2115-M-032-006-MY3.
This work was carried out in the framework of the International Research Network ``ReaDiNet'' jointly funded by CNRS and NCTS.}

%%%%%%%%%%%%%%%%%%%%%
%%%%%%%%%%%%%%%%%%%%%

\end{document}